\documentclass[12pt]{amsart}
\usepackage{amscd}
\usepackage{amssymb,amsmath}
\usepackage{thmtools}
\usepackage{hyperref}
\usepackage{mathrsfs}
\usepackage{graphicx}
\usepackage{enumitem}
\usepackage{romannum}
\usepackage{mathtools}
\usepackage[utf8]{inputenc}
\usepackage{listings}
\usepackage[capitalise]{cleveref}

\usepackage{lipsum,eso-pic,xcolor}
\usepackage{lineno}

\usepackage{tikz}
\usetikzlibrary{calc}
\usetikzlibrary{arrows,decorations.pathmorphing,backgrounds,positioning,fit}
\usetikzlibrary{positioning}
\hypersetup{
    colorlinks=true,
    linkcolor=blue,
    hypertexnames=false,
}

\usetikzlibrary{trees}
\usetikzlibrary{arrows}

\def\supp{\operatorname{supp}}

\def\reg{\operatorname{reg}}
\def\deg{\operatorname{deg}}

\def\max{\operatorname{max}}

\DeclarePairedDelimiter\ev{\langle}{\rangle}

\newcommand{\m}{\mathfrak m}

\newcommand{\N}{\mathcal{N}}

\newcommand{\G}{\mathcal{G}}

\newcommand{\PP}{\mathcal{P}}

\newtheorem{proposition}{Proposition}[section]
\newtheorem{lemma}[proposition]{Lemma}

\newtheorem{theorem}[proposition]{Theorem}
\newtheorem{definition}[proposition]{Definition}
\newtheorem{remark}[proposition]{Remark}
\newtheorem{notation}[proposition]{Notation}
\newtheorem{example}[proposition]{Example}
\newtheorem{question}[proposition]{Question}

\newcommand{\basegraph}[2]{%
\begin{tikzpicture}[scale=1.15]
  \coordinate (v1) at (0,1.2);
  \coordinate (v2) at (1.8,1.2);
  \coordinate (v3) at (1.8,0);
  \coordinate (v4) at (0,0);
 \foreach \i/\j in {#1}{
    \ifnum\i=1
      \ifnum\j=2
        \draw[red, very thick] (v\i) -- (v\j);
      \else
        \draw (v\i) -- (v\j);
      \fi
    \else
      \draw (v\i) -- (v\j);
    \fi
  }
  \foreach \i in {1,2,3,4} \fill (v\i) circle (2.8pt);
  \node[above left]  at (v1) {$x_1$};
  \node[above right] at (v2) {$x_2$};
  \node[below right] at (v3) {$x_3$};
  \node[below left]  at (v4) {$x_4$};
  \node at (0.9,-0.75) {#2};
\end{tikzpicture}}

\advance\headheight1.15pt

\newtheorem{innercustomthm}{Theorem}
\newenvironment{customthm}[1]
  {\renewcommand\theinnercustomthm{#1}\innercustomthm\itshape}
  {\endinnercustomthm}

\begin{document}
	
	\pagenumbering{arabic}
	
	\title[Componentwise Linearity of Support-Two Monomial Ideals]{Componentwise linearity, Fr\"oberg's analogue, and classification of linear support-two monomial ideals} 
	
	\author[Manohar Kumar]{Manohar Kumar$^a$}
	\address{$^a$Department of Mathematics, Indian Institute of Technology Madras, Chennai, Tamil Nadu, INDIA - 600036.}
	\email{manhar349@gmail.com}

	\author[Kamalesh Saha]{Kamalesh Saha$^b$}
	\address{$^b$Department of Mathematics, SRM University-AP, Amaravati 522240, Andhra Pradesh, India.}
	\email{kamalesh.saha44@gmail.com; kamalesh.s@srmap.edu.in}
	
	%\thanks{$^b$ Corresponding Author}
	
	\thanks{AMS Classification 2020: 13D02, 13F55, 05E40, 05C25}

	%\keywords{Regularity of }

	\begin{abstract}
		In this paper, we investigate the componentwise linear property of support-two monomial ideals. Our first main result shows that if $I$ is a support-two monomial ideal, then its underlying simple graph $G_I$ is co-chordal; equivalently, by Fr\"{o}berg's theorem, $\sqrt{I}$ admits a linear resolution. This phenomenon is quite rare for general monomial ideals. In fact, there exist monomial ideals with linear resolutions whose radicals fail to have linear resolutions, even when the radical is the edge ideal of a graph.\par 

For any monomial ideal $I$, one always has $\mu(I)\geq \mu(\sqrt{I})$. In the literature, support-two monomial ideals satisfying $\mu(I)=\mu(\sqrt{I})$ are of special interest, as they include edge ideals of simple graphs, weighted oriented graphs, edge-weighted graphs, and vertex-weighted graphs. We refer to such ideals as \emph{minimal support-two monomial ideals}. We explicitly characterize all minimal support-two monomial ideals, as well as their powers, that admit linear resolutions.  Next, we classify the linearity of non-minimal ones. Consequently, we obtain a complete classification of linear support-two monomial ideals, which shows that the property of being linear does not depend on the characteristics of the base field for support-two monomial ideals. 
\end{abstract}

\maketitle
    
\section{Introduction}
The study of ideals having linear resolutions has received significant attention over the years because of their importance in areas such as commutative algebra, algebraic geometry, combinatorial topology, and combinatorics. These ideals are well-behaved, admit several algebraic interpretations, and their minimal free resolutions take a very special form. A remarkable result of Eagon and Reiner states that the Stanley-Reisner ideal of a simplicial complex is Cohen-Macaulay if and only if its Alexander dual has a linear resolution \cite{ev98}. Inspired by the significance of this class of ideals, Herzog and Hibi \cite{hh99} introduced in 1999 the next best possible class beyond ideals with linear resolution, known as componentwise linear ideals. Let $I \subseteq R=\mathbb{K}[x_1,\ldots,x_n]=\bigoplus_{d\geq0} R_d$ be a graded ideal over a field $\mathbb{K}$. For each integer $d\geq0$, denote by $I_{\langle d\rangle}$ the ideal generated by all homogeneous elements of degree $d$ in $I$, that is, $I_{\langle d\rangle}=(I_d)$. The ideal $I$ is said to be \emph{componentwise linear} if, for every $d\geq0$, the ideal $I_{\langle d\rangle}$ admits a linear resolution. Thus, when $I$ is generated in a single degree, the notion of being componentwise linear
coincides with having a linear resolution.\par 

Componentwise linear ideals have attracted considerable interest in commutative algebra because of their nice properties and algebraic interpretation. Herzog and Hibi \cite{hh99} strengthened the Eagon–Reiner theorem by showing that a Stanley-Reisner ideal $I$ is sequentially Cohen-Macaulay if and only if its Alexander dual $I^{\vee}$ is componentwise linear. One of the primary motivations to study componentwise linear ideals comes from the literature on Koszul algebras and Koszul modules. Specifically, Herzog-Hibi \cite{hh99} and Iyengar-R\"omer \cite{ir09} proved that $R/I$ is Koszul if and only if $I$ is componentwise linear. Another important result of Herzog, Reiner, and Welker \cite{hrw99} states that if a graded ideal $I \subseteq R$ is componentwise linear and contains no linear forms, then $I$ is Golod. Thanks to Harima and Watanabe \cite{hw15}, the class of completely $\mathfrak{m}$-full ideals coincides with the class of componentwise linear ideals in a polynomial ring over an infinite field. Some references on the theory of componentwise linearity  \cite{ds25}, \cite{hv22}, \cite{hh99}, \cite{hhm22}, \cite{nr15}, \cite{nms21}.\par

In this paper, we study the componentwise linearity of support-two monomial ideals. The objects of our study are monomial ideals whose minimal generating sets consist of monomials supported on exactly two variables. More precisely, a monomial ideal $I \subseteq R$ is said to be a \emph{support-two monomial ideal} if every monomial in $\G(I)$ is of the form $x_i^{a_i} x_j^{b_j}$ for two distinct $i,j\in[n]$ and $a_i, b_j \ge 1$, where $\G(I)$ denotes the minimal set of monomial generators of $I$. A classical and well–studied subclass of support-two monomial ideals arises from graphs, namely, edge ideals of graphs. Let $G$ be a simple graph with vertex set $V(G) = \{x_1, \ldots, x_n\}$ and edge set $E(G)$. The \emph{edge ideal} of $G$, denoted by $I(G)$, is defined as $I(G) := ( x_i x_j \mid \{x_i, x_j\} \in E(G))$. Hence, one can see that if $I$ is a support-two monomial ideal, then its radical $\sqrt{I}$ is the edge ideal of a simple graph, which we denote by $G_{I}$ and refer to as the \emph{base graph} of $I$. In other words, $\sqrt{I}=I(G_{I})$. Moreover, for a monomial ideal $I$, we always have $\mu(I)\geq \mu(\sqrt{I})$, where $\mu(I)=|\G(I)|$. We call $I$ a \textit{minimal} support-two monomial ideal if $I$ is a support-two monomial ideal with $\mu(I)= \mu(\sqrt{I})$.\par

The study of edge ideals of graphs, introduced by Villarreal \cite{v90} in 1990, marked an important milestone in the interaction between commutative algebra and graph theory. Since then, edge ideals have been extensively investigated and have become a central topic of research due to their rich algebraic structure and numerous applications. Over time, several classes of ideals associated with graphs have been introduced as generalizations of edge ideals, motivated by their appearance in other areas of Mathematics. Examples include edge ideals of weighted oriented graphs, which arise naturally in the study of Reed-Muller-type codes (see \cite{hlmrv19, prt19}), edge ideals of edge-weighted graphs (see \cite{ps13}), and edge ideals of vertex-weighted graphs (see \cite{k20}). It is worth noting that all these ideals belong to the class of minimal support-two monomial ideals. Consequently, support-two monomial ideals form a large and significant class within the theory of monomial ideals. However, almost nothing is known about these ideals in full generality, and many basic questions remain open. A systematic study of this class could lead to a rich theory unifying several existing classes of monomial edge ideals. As far as we know, the literature contains only one general study on support-two monomial ideals, in which the equality of ordinary and symbolic powers has been investigated for some subclasses (see \cite{bdk25}).\par 

One of the most celebrated results in the literature of edge ideals of simple graphs is Fr\"{o}berg's theorem \cite{f90}, which states that an edge ideal $I(G)$ has a linear resolution if and only if the complement of $G$ is chordal. In \cite{bds23}, the authors classified all weighted-oriented edge ideals that admit linear resolutions. In addition, \cite{kns25} provides some necessary and sufficient conditions for such ideals to be componentwise linear. In contrast to these results and Fr\"oberg's theorem, two fundamental questions that arise in a much broader sense are the following.
\begin{question}\label{Q1}
    Classify all types of monomial edge ideals associated with graphs that admit linear resolutions. More generally, classify minimal support-two monomial ideals with linear resolutions.
\end{question}
\noindent A combinatorial classification of componentwise linear monomial edge ideals is considerably more difficult. In fact, the property of being componentwise linear may depend on the base field. However, inspired by Fr\"oberg's theorem, it is natural to seek necessary conditions for the componentwise linearity of support-two monomial ideals in terms of underlying simple graphs.
\begin{question}\label{Q2}
     Let $I$ be a graph-associated monomial edge ideal that is componentwise linear. More generally, suppose $I$ is a componentwise linear support-two monomial ideal (not necessarily minimal). What can be said about the underlying simple graph $G_{I}$?
\end{question}
\par 
In this paper, we completely answer \Cref{Q1} and \ref{Q2}. The following theorem constitutes the key and somewhat surprising result of this article, establishing a connection between the componentwise linearity of support-two monomial ideals and the chordality of the complements of their underlying graphs. 

\begin{customthm}{\ref{thmcochordal}}
    Let $I\subseteq R$ be a support-two monomial ideal. If $I$ is componentwise linear, then the complement of $G_I$ is chordal (equivalently, $\sqrt{I}$ has a linear resolution).
\end{customthm}
\noindent Let us explain why the above result is both surprising and important. For a componentwise linear monomial ideal $I$, it is quite rare that the radical $\sqrt{I}$ is again componentwise linear. In fact, there are many examples where $I$ even has a linear resolution, $\sqrt{I}$ is the edge ideal of a simple graph, yet $\sqrt{I}$ fails to be componentwise linear. For instance, if we take $I = I(C_5)^k$ for some $k \geq 2$, then $I$ has a linear resolution by \cite[Theorem 5.2]{bht15}, whereas $\sqrt{I} = I(C_5)$ does not. However, \Cref{thmcochordal} together with Fr\"oberg's theorem shows that if $I$ is a componentwise linear support-two monomial ideal, then \(\sqrt{I}\) admits a linear resolution. The proof of \Cref{thmcochordal} is quite technical and relies on a careful interplay between algebraic and combinatorial tools, including componentwise linearity, the linearity of edge ideals, polarization, and the exponents of variables in the generators.

Our next goal is to classify all support-two monomial ideals that have linear resolutions. We divide the classification into two cases, considering minimal ideals first and then non-minimal ideals. The latter case is more challenging and relies on the classification obtained in the minimal case. Our second main result not only provides a complete characterization of minimal support-two monomial ideals with linear resolutions, but also characterizes when their powers have linear resolutions. More generally, we prove the following.

\begin{customthm}{\ref{equivalent}}
Let $I\subseteq R$ be a minimal support-two monomial ideal with $\alpha(I) > 2$. Then the following conditions are equivalent:
\begin{enumerate}
    \item $I$ has a linear resolution.
    \item $I=(x_i x_j^{a_j} \mid i\in [n]\setminus \{j\})$ for some $1 \le j \le n$ and $a_j>1$.
    \item $I^k$ has a linear resolution for all $k \ge 1$.
    \item $I^k$ has a linear resolution for some $k \ge 1$.
    \item $I^k$ has linear quotients for all $k \ge 1$.
\end{enumerate}
\end{customthm}
\noindent Here, $\alpha(I)$ denotes the initial degree of $I$. The case $\alpha(I)=2$ corresponds to edge ideals of simple graphs, and the conditions under which $I$ and all its powers have linear resolutions and linear quotients property are characterized in \cite[Theorem 1]{f90}, \cite[Theorem 3.2]{hhz104} and \cite[Theorem 2.6]{a15}. However, the classification of edge ideals whose certain powers have linear resolutions is still open. Finally, using the structure of minimal support-two monomial ideals having linear resolutions, together with other techniques, we give an explicit classification of non-minimal linear support-two monomial ideals as follows:
\begin{customthm}{\ref{thm:linearresolution}}
Let $I$ be a non-minimal support-two monomial ideal, i.e. $\mu(I)>\mu(\sqrt{I})$. Then $I$ has a linear resolution if and only if $I$ is one of the following ideals (up to the permutation of variables)
\[
\begin{aligned}
\textup{(1)}\quad
I &=
(x_1^{a_1}x_2^{b_1}, \ldots, x_1^{a_1-k}x_2^{b_1+k});\\[1ex]
\textup{(2)}\quad
I &=
(x_1^{a_1}x_2^{b_1}, \ldots, x_1x_2^{b_1+k},
x_2^{b_1+k}x_3, \ldots, x_2^{b_1+k}x_n);\\[1ex]
\textup{(3)}\quad
I &=
(x_1^{k}x_2,\ldots,x_1x_2^{k},
x_3x_1^{k},\ldots,x_sx_1^{k},
x_{s+1}x_2^{k},\ldots,x_nx_2^{k});\\[1ex]
\textup{(4)}\quad
I &=
(x_1^{k}x_2,\ldots,x_1x_2^{k},
x_3x_1^{k},\ldots,x_nx_1^{k},
x_{3}x_2^{k},\ldots,x_nx_2^{k});\\[1ex]
\textup{(5)}\quad
I &=
(x_1^{k}x_2,\ldots,x_1x_2^{k},
x_3x_1^{k},\ldots,x_px_1^{k},
x_{p+1}x_1^{k},\ldots,x_nx_1^{k},\\
&\qquad\qquad\qquad\qquad\qquad\qquad\qquad\qquad\quad x_{p+1}x_2^{k},\ldots,x_nx_2^{k});\\[1ex]
\textup{(6)}\quad
I &=
(x_1^{k}x_2,\ldots,x_1x_2^{k},
x_3x_1^{k},\ldots,x_px_1^{k},
x_{p+1}x_1^{k},\ldots,x_sx_1^{k},\\
&\qquad\qquad\qquad\qquad\qquad
x_{p+1}x_2^{k},\ldots,x_sx_2^{k},
x_{s+1}x_2^{k},\ldots,x_nx_2^{k}).
\end{aligned}
\]
\end{customthm}
The paper is organized as follows. In \Cref{secpreli}, we recall the necessary definitions, basic notions, and known results used throughout the paper. In \Cref{seccochordal}, we first establish several propositions and technical lemmas and then prove our main result \Cref{thmcochordal}. In \Cref{powers-minimal-support}, we again build the required setup through lemmas and propositions before proving our main result, and conclude by giving a complete characterization of the linear resolutions of powers of minimal support-two monomial ideals in \Cref{equivalent}. We observe from \Cref{equivalent} that, for a minimal support-two monomial ideal to admit a linear resolution, its base graph must be a star graph. In \Cref{sec5}, we establish a series of lemmas, and using these lemmas together with \Cref{equivalent}, we give a structural classification of non-minimal support-two monomial ideals admitting linear resolutions; see \Cref{thm:linearresolution}. Hence, we obtain a complete classification of linear support-two monomial ideals by Fr\"oberg's theorem, \Cref{equivalent}, and \Cref{thm:linearresolution}. 
%We conclude the paper by posing an open question concerning the linearity of support-two monomial ideals that are not minimal. 

\section{Preliminaries}\label{secpreli}
In this section, we recall some necessary prerequisites, which are used to describe our work and establish our results.\par 

Let $G=(V(G), E(G))$ be a simple graph. For a vertex $x\in V(G)$, the \textit{neighbourhood} of $x$ in $G$, denoted by $\N_{G}(x)$, is defined as $\mathcal{N}_{G}(x)= \{y \in V(G) \mid \{x,y\} \in E(G)\}$. For a subset $A\subseteq V(G)$, we denote the \textit{induced} subgraph of $G$ on $A$ by $G[A]$, in particular, $G[A]$ is a graph with $V(G[A])=A$ and $E(G[A])=\{e\in E(G)\mid e\subseteq A\}$. Again, by $G\setminus A$, we mean the induced subgraph $G[V(G)\setminus A]$. For simplicity of notation, we write $G\setminus x$ to denote the graph $G\setminus \{x\}$ for any vertex $x\in V(G)$. The \textit{complement} of a simple graph $G$, denoted by $G^c$, is a simple graph such that $V(G^c)=V(G)$ and $E(G^c)=\{ \{u,v\}\subseteq V(G)\mid \{u,v\}\not\in E(G)\}$.

\begin{definition}{\rm
A \textit{cycle} of length $n$ (or $n$-\textit{cycle}), denoted by $C_n$, is a connected graph on the vertex set $V(C_n)=\{x_1,\ldots,x_n\}$ such that $E(C_n)=\{\{x_{i},x_{i+1}\}\mid 1\leq i\leq n \text{ and } x_{n+1}=x_1\}$. A \textit{path} of length $n-1$ (or a \textit{path} on $n$ vertices), denoted by $P_n$, is a connected graph with $V(C_n)=\{x_1,\ldots,x_n\}$ and $E(C_n)=\{\{x_{i},x_{i+1}\}\mid 1\leq i\leq n-1\}$. A simple graph $G$ is called {\it chordal} if $G$ has no cycle of length greater than three as an induced subgraph. A graph $G$ is called {\it co-chordal} if $G^c$ is chordal. A \textit{complete} graph on $n$ vertices, denoted by $K_n$, is defined as $V(K_n)=\{x_1,\ldots,x_n\}$ and $E(K_n)=\{\{x_i,x_j\}\mid 1\leq i<j\leq n\}$.
}
\end{definition}

\begin{definition}{\rm Let $I$ be a graded ideal of $R$ and $\m$ is the unique homogeneous maximal ideal of $R$. The \textit{Castelnuovo-Mumford regularity} (or simply \textit{regularity}) of $I$, denoted by $\reg(I)$, is defined as follows 
 \begin{align*}
 \reg(I): &=  \max \{j - i \mid \beta_{i,j}(I) \neq 0\} \\
         &= \max\{j+i \mid H_{\m}^i(I)_j \neq 0\},   
 \end{align*}
  where $\beta_{i,j}(I)$ is the $(i,j)^{th}$ graded Betti number of $I$ and $H_{\m}^i(I)_j$ denotes the $j^{th}$ graded component of the $i^{t h}$ local cohomology module $H_{\m}^i(I)$.
  }
\end{definition}
 Let us recall the definitions of linear resolution and componentwise linear ideals.
\begin{definition}{\rm
Let $I$ be a graded ideal in $R$. If the minimal free resolution of $I$ is given by
$$0 \to R(-(d+p))^{\beta_p} \to \cdots \to R(-(d+1))^{\beta_1} \to R(-d)^{\beta_0} \to I \to 0,$$
then we say $I$ has a \textit{linear} resolution. In other words, $I$ is said to be linear if $\reg(I)=\alpha(I)$, where $\alpha(I)$ denotes the minimum degree of a homogeneous element belongs to $I$. Note that if $I$ has a linear resolution, then $I$ is equigenerated (generated in a single degree).\par 

A graded ideal $I$ of $R$ is said to be \textit{componentwise linear} if for each $d\geq 0$, the ideal generated by all homogeneous element of degree $d$, denoted by $I_{\ev{d}}$, has a linear resolution over $R$, i.e., $\reg(I_{\langle d \rangle})=d$ for all $d\geq 0$. Note that if $I$ is equigenerated, then linearity and componentwise linearity are equivalent.
   }
\end{definition}
 
\begin{definition}{\rm
    A monomial ideal $I \subseteq R$ has \textit{linear quotient} property if there exists an order $ u_1 < \cdots < u_m$ on the minimal monomial generating set $\mathcal{G}(I)$ of $I$ such that the colon ideal $((u_1, \ldots, u_{i-1}) : u_i)$ is generated by a subset of variables, for $i = 2, \ldots, m$. By \cite[Theorem 8.2.15]{hh11}, if $I$ has linear quotients, then $I$ is componentwise linear.
    }
\end{definition}

Now, let us state the well-known Fr\"{o}berg's theorem \cite[Theorem 1]{f90}, which has been used frequently in this paper.\medskip

\noindent{\bf Fr\"{o}berg's Theorem.} Let $G$ be a simple graph and $I(G)$ be its edge ideal. Then $I(G)$ has a linear resolution if and only if $G^c$ is chordal.
\medskip

For a monomial $u\in R$, define its {\it support} as $\supp(u):= \{x_i \mid x_i \text{ divides } u\}$. We denote the degree of $u$ by $\deg(u)$.

\begin{definition}{\rm
Let $m=x_1^{a_1}\cdots x_n^{a_n}$ be a monomial in $R$. Then the {\it polarization} of $m$ is defined to be the square-free monomial
\begin{equation*}
    \PP(m)=x_{11}x_{12} \cdots x_{1a_1} \cdots x_{n1}x_{n2}\cdots x_{n a_n}
\end{equation*}
in the polynomial ring $\mathbb{K}[x_{i j} \mid 1 \leq i \leq n, 1 \leq j \leq a_i].$ Let $I \subseteq R $ be a monomial ideal with $\mathcal{G}(I)=\{m_1, \ldots, m_u\}$ and $m_i=\prod_{j=1}^{n}x_j^{a_{ij}}$, where each $a_{ij} \geq 0$ for $i\in [u]=\{1, \ldots, u\}$. Then the {\it polarization} of $I$, denoted by $I^{\PP}$, is defined as 
\begin{equation*}
    I^{\PP}=(\PP(m_1), \ldots, \PP(m_u)),
\end{equation*}
which is a square-free monomial ideal in the polynomial ring $R^{\PP}=\mathbb{K}[x_{j1}, \ldots, x_{ja_j} \mid j=1, \ldots, n ]$, where $a_j=\max\{a_{i j} \mid i\in [u] \}$ for any $1 \leq j \leq n.$
}
\end{definition}

\begin{lemma}\cite[Proposition $1$]{nms21}\label{compntpolariz}
   Let $I$ be a monomial ideal in $R=\mathbb{K}[x_1,\ldots,x_n]$. Then $I$ is a componentwise linear ideal if and only if $I^{\PP}$ is a componentwise linear ideal. 
\end{lemma}

\begin{lemma}\cite[Corollary 8.2.14]{hh11}\label{maximum}
 Let $I \subseteq R$ be a componentwise linear ideal. Then $\reg(I)=d(I)$, where $d(I)$ is the highest degree of a minimal generator $I$.   
\end{lemma}

The following lemma is standard in the literature and will be used without explicit reference.

\begin{lemma} \label{reg-sum}
Let $S_1=\mathbb{K}[x_1,\ldots,x_m]$ and $S_2=\mathbb{K}[x_{m+1},\ldots,x_n]$ be two polynomial rings, $I_1 \subset S_1$ and $I_2 \subset S_2$ be two non-zero graded ideals. Then $$\reg(I+J)=\reg(I)+\reg(J)-1.$$
Again, for any homogeneous element $f\in R$ and a homogeneous ideal $I\subseteq R$, we have
$$\reg(fI)=\deg(f)+\reg(I).$$
\end{lemma}

The following lemma is widely used in the literature and is known as the regularity lemma.
\begin{lemma}\cite[Corollary 20.19]{e95} \label{lem:regularity-lemma}
		Let $0\rightarrow A \rightarrow B \rightarrow C \rightarrow 0$ be a short exact sequence of finitely generated graded $R$-modules. Then 
		\begin{enumerate}
    \item $\reg B \leq \max\{\reg A, \reg C\},$
    
    \item $\reg A \leq \max\{\reg B, \reg C + 1\},$
    
    \item $\reg C \leq \max\{\reg A - 1, \reg B\}.$
\end{enumerate}
\end{lemma}
\section{Necessary conditions for componentwise linearity}\label{seccochordal}

In this section, we study necessary conditions for a support-two monomial ideal to be componentwise linear. Our main result says that if a support-two monomial ideal $I$ is componentwise linear, then the underlying simple graph $G_I$ is co-chordal, i.e., $\sqrt{I}=I(G_I)$ has a linear resolution. This implication, however, does not hold for arbitrary monomial ideals whose radicals are edge ideals of graphs.

\begin{definition}{\rm
Let us recall support-two monomial ideals and related notions:
\begin{enumerate}
    \item  Let $I\subseteq R$ be a monomial ideal. Then $I$ is said to be a \textit{support-two} monomial ideal if the support of each minimal monomial generator of $I$ has cardinality two.

    \item Let $I$ be a support-two monomial ideal. Then we have $\sqrt{I}=I(G_I)$ for some simple graph $G_I$. We call $G_I$ the \textit{base} graph (or \textit{underlying} graph) of $I$.

    \item Also, note that $\mu(I)\geq \mu(\sqrt{I})=\vert E(G_I)\vert$. Thus, a support-two monomial ideal $I$ with the base graph $G_{I}$ having minimum possible generators is of special interest. We say that $I$ is a \textit{minimal} support-two monomial ideal if $\mu(I)=\mu(\sqrt{I})$.
\end{enumerate}
}    
\end{definition}

\begin{example}{\rm
    Let $I=( x_{1}^{5}x_{2}^2, x_{1}^{3}x_{2}^{4})$. Then $I$ is a support-two monomial ideal with the base graph $P_2$, but $I$ is not minimal as $\mu(I)=2>1=\mu(\sqrt{I})$. Now, if we consider edge ideals of simple graphs, weighted graphs, weighted oriented graphs, and edge-weighted graphs, then all of these form subclasses of minimal support-two monomial ideals.
    }
\end{example}

\begin{notation}\label{notation}{\rm
   Let $I$ be a support-two monomial ideal in $R$. For each $1 \le i \le n$, 
we denote by $w_i(I)$ (resp. $\alpha_i(I)$) the highest (resp. lowest) 
exponent of $x_i$ appearing in $\G(I)$ that is greater than or equal to one.
}
\end{notation}

\begin{proposition}\cite[Proposition 3.1]{kns25}\label{proplin}
	Let $I\subseteq R$ be a monomial ideal. For a variable $x$ in $R$, consider the ideal $I'$ such that $\mathcal{G}(I')=\mathcal{G}(I) \setminus \{u \in \mathcal{G}(I) \mid x \mid u\}$. Then the following holds:
\begin{enumerate}
    \item  If $I$ has a linear resolution, then $I'$ also has a linear resolution.   
\item If $I$ is componentwise linear, then $I'$ is also componentwise linear. 
\end{enumerate}    
\end{proposition}

\begin{proposition}\label{thmind4}
	Let $I\subseteq R$ be a support-two monomial ideal. If $I$ is componentwise linear, then $(G_I)^c$ has no induced $4$-cycle.  
	\end{proposition}
	\begin{proof}
		Suppose $(G_{I})^c$ has an induced $4$-cycle $C_4$ with the vertex set $\{x_1,x_2,x_3,x_4\}$ and the edge set $\{\{x_1,x_2\},\{x_2,x_3\},\{x_3,x_4\},\{x_1,x_4\}\}$. Then we have $\{\{x_1,x_3\},\{x_2,x_4\}\} \subseteq E(G_I)$, which forms an induced matching in $G_I$. Now, we consider two monomial ideals $I_1$ and $I_2$ such that $\G(I_1)=\{f\in \G(I)\mid \supp(f)=\{x_1,x_3\}\}$ and $\G(I_2)=\{g\in \G(I)\mid \supp(g)=\{x_2,x_4\}\}$. By the definition of regularity, we have $\reg(I_1)\geq d(I_1)$ and $\reg(I_2)\geq d(I_2)$. Thus, $\reg(I_1+I_2)=\reg(I_1)+\reg(I_2)-1\geq d(I_1)+d(I_2)-1$. Let $I'=I_1+I_2$. Then $\reg(I')> d(I')=\max\{d(I_1),d(I_2)\}$ as both $d(I_1),d(I_2)\geq 2$. Since $\{\{x_1,x_3\},\{x_2,x_4\}\}$ forms an induced matching in $G_I$, one can verify that 
        $$\G(I')=\G(I)\setminus \{u \in \G(I) \mid x_i \text{ divides } u \text{ for some } 5\leq i\leq n\}.$$ 
        Thus, by \Cref{proplin}, $I'$ is componentwise linear as $I$ is componentwise linear. This gives a contradiction due to \Cref{maximum} and the fact that $\reg(I')>d(I')$. Hence, $(G_{I})^c$ has no induced $4$-cycle.
	\end{proof}

\begin{lemma}\label{path-upper-bound}
   Let $I \subseteq R$ be a support-two monomial ideal such that $\alpha_{i}(I)=w_{i}(I)$ for all $1 \leq i \leq n$. If $G_{I}$ is co-chordal, then we have
   $$\reg(I) \leq \sum_{i=1}^n \alpha_{i}(I)-\vert V(G_I) \vert +2.$$
\end{lemma}
\begin{proof}
Observe that due to the given condition, $I$ is a minimal support-two monomial ideal; in fact, $I$ is the edge ideal of a vertex-weighted graph. We proceed by mathematical induction on the sum of the exponents of variables $m(I)=\sum_{i=1}^n \alpha_i(I)$. Note that $m(I)\geq n$, where $n$ is the number of variables involved in $\G(I)$, i.e., $n=|V(G_I)|$. So, the base case is $m(I)=n$. Note that $m(I)=n$ if and only if $\alpha_i(I)=1$ for all $i$. In this case, we have $I=I(G_I)$, and thus, by Fr\"{o}berg's theorem, we have $\reg(I)=2$. Now, let us assume $m(I)>n$. Then there exists a vertex $x_i$ for which $\alpha_i(I)>1$. Without loss of generality, suppose the exponent of the variable $x_1$ in $\G(I)$ is greater than one, i.e., $\alpha_{1}(I)>1$. Then
$$(I:x_1)=\left(x_1^{\alpha_1(I)-1}x_{i}^{\alpha_{i}(I)} \mid x_1^{\alpha_1(I)}x_{i}^{\alpha_{i}(I)} \in \mathcal{G}(I)\right)+\left(u\in\G(I)\mid x_1\nmid u\right).$$
It is easy to see that $I'=(I:x_1)$ is again a support-two monomial ideal with $\alpha_i(I')=w_{i}(I')$ for all $1\leq i\leq n$ and the base graph of $I'$ is $G_{I}$ itself. Since $m(I')=\sum_{i=1}^{n}\alpha_i(I')=m(I)-1$ and $G_{I'}=G_{I}$ is co-chordal, using the induction hypothesis, we get $\reg(I') \leq \sum_{i=1}^n \alpha_i(I) -\vert V(G) \vert +1$. Now, consider the ideal $I''=(u\in\G(I)\mid x_1\nmid u)$. Then we have $(I,x_1)=(I'',x_1)$. Observe that $I''$ is a support-two monomial ideal with the base graph $G_{I''}=G_{I}\setminus x_1$ and $\alpha_{i}(I'')=w_{i}(I'')$ for all $2\leq i\leq n$. Since $G_I$ is co-chordal, $G_{I''}$ is so. It is obvious that $m(I'')\leq m(I)-2$ as $\alpha_{1}(I)\geq 2$. Thus, applying the induction hypothesis, we have 
$$\reg(I,x_1)=\reg(I'') \leq \sum_{i=1}^n \alpha_{i}(I) - \vert V(G_I) \vert +1.$$
\noindent Now, we consider the following exact sequence
$$0 \rightarrow R/(I:x_1)(-1) \rightarrow R/I \rightarrow R/(I,x_1) \rightarrow 0.$$
  Then by applying \Cref{lem:regularity-lemma}(1) on the above exact sequence, we have $$\reg(R/I)\leq \max\{\reg(R/(I:x_1))+1, \reg(R/(I,x_1))\}.$$
\noindent Hence, we have $\reg(I) \leq \sum_{i=1}^n \alpha_i(I)-\vert V(G_I) \vert +2$.

% Hence, \textcolor{red}{by \cite[Corollary 3.3(ii)]{chhktt19}}, we have $\reg(I) \leq \sum_{i=1}^n \alpha_i(I)-\vert V(G_I) \vert +2$.
\end{proof}

\begin{proposition}\label{lemcochord}
Let $I\subseteq R$ be a support-two monomial ideal such that $\alpha_i(I)=w_i(I)$ for all $1 \leq i \leq n$. Suppose $(G_{I})^c$ is a cycle of length greater than $3$. Then we have 
$$ \reg(I) = \sum_{i=1}^n \alpha_i(I)-\vert V(G_I) \vert +3.$$
 \end{proposition}
 \begin{proof}
We proceed by induction on $m(I)=\sum_{i=1}^{n}\alpha_i(I)$. Since $m(I)\geq n=|V(G_I)|$, the base case is $m(I)=n$, which implies $I=I(G_I)$. In this case, it follows from Fr\"{o}berg's theorem that $\reg(I)\geq 3$. If $(G_I)^c=C_4$, then $G_I$ is a disjoint union of two edges, which gives $\reg(I)=3$. Again, if $(G_I)^c=C_n$ for some $n\geq 5$, then we see that $G_I$ is a gap-free and claw-free graph. Thus, $\reg(I)= 3$ by Nevo's result \cite[Theorem 1.2(2)]{n11}. Therefore, the inductive statement is true in the base case. Now, assume $m(I)>n$. Then there exists $k \in [n]$ such that $\alpha_{k}(I)>1$. Then, proceeding as in \Cref{path-upper-bound}, we see that $I'=(I:x_k)$ is a support-two monomial ideal with $m(I')=m(I)-1$ and $G_{I'}=G_{I}$. Thus, by the induction hypothesis, we have $\reg(I') = \sum_{i=1}^n \alpha_i(I) -\vert V(G_I) \vert +2$. Now, let us focus on the ideal $I''=(g\in \G(I)\mid x_{k}\nmid g)$. Then one can see that $I''$ is a support-two monomial ideal with the base graph $G_{I''}=P_{n-1}^c$. Since path graph is a chordal graph, $G_{I''}$ is co-chordal, and thus, by \Cref{path-upper-bound}, we have
     \begin{align*}
         \reg(I,x_k)&=\reg(I'',x_k)\\
         &\leq \sum_{i\in [n]\setminus \{k\}} \alpha_{i}(I'') -\vert V(G_{I''}) \vert +2\\
         &\leq\sum_{i=1}^n \alpha_i(I)-\alpha_{k}(I)-(\vert V(G_{I})\vert-1)+2\\
         &<\sum_{i=1}^n \alpha_i(I)-\vert V(G_{I})\vert+2=\reg(I:x_k) \quad (\text{since }\alpha_k(I)>1).
     \end{align*}
 \noindent Now, we consider the following exact sequence
$$0 \rightarrow R/(I:x_k)(-1) \rightarrow R/I \rightarrow R/(I,x_k) \rightarrow 0.$$
  Then by applying \Cref{lem:regularity-lemma}(1) \& (2) on the above exact sequence, we have 
  $$\reg(R/I)= \reg(R/(I:x_k))+1,  \text{ because } \reg(R/(I,x_k))+1 < \reg(R/(I:x_k))+1.$$
Therefore, we have $\reg(I)= \reg(I:x_k)+1=\sum_{i=1}^n \alpha_i(I)-\vert V(G_I) \vert +3$.
 \end{proof}

 Now, we are going to prove our first main result of this paper.
\begin{theorem}\label{thmcochordal}
	Let $I\subseteq R$ be a support-two monomial ideal. If $I$ is componentwise linear, then the complement of $G_I$ is chordal (equivalently, $\sqrt{I}$ has a linear resolution).  
\end{theorem}
	\begin{proof}
		For a clearer understanding of the proof, the reader is encouraged to be thoroughly familiar with \Cref{notation}. We proceed by induction on $m(I)=\sum_{i=1}^n w_i(I)$. It is clear that $m(I)\geq n$, where $n$ is the number of vertices of the underlying simple graph $G_{I}$ of $I$. So, the base case is $m(I)=n$, which is equivalent to saying $w_i(I)=1$ for all $i\in [n]$, i.e., $I=I(G_I)$. In this case, componentwise linearity of $I$ implies $I$ has a linear resolution, since $I$ is equigenerated in degree $2$. Hence, $(G_{I})^c$ is chordal by Fr\"{o}berg's theorem. Now, let us assume $m(I)>n$. Then there exists at least one vertex $x_i\in V(G_I)$ with $w_i(I)>1$. From \Cref{thmind4}, it follows that $(G_{I})^c$ has no induced cycle of length $4$. Now, suppose $(G_{I})^c$ is not chordal, and it contains a cycle $C$ of length greater than or equal to $5$. Let $G_{I}[V(C)]$ be the induced subgraph of $G_{I}$ on the vertex set $V(C)$. Then $G_{I}[V(C)]=C^c$, the complement of the cycle $C$. For simplicity of notation, we write $G_I=G$. We consider the following possible cases:\par 
  
		\noindent {\bf Case-I.} Suppose $G$ contains a vertex $x_j$ other than the vertices of $C$. Then $I'$ is componentwise linear by \Cref{proplin}(2), where $I'=(g\in \G(I)\mid x_j\nmid g)$. Note that $I'$ is again a support-two monomial ideal with the underlying simple graph $G\setminus x_j$ and $m(I')<m(I)$. Therefore, by the induction hypothesis, $(G\setminus x_{j})^c$ is chordal. However, since $x_j\not\in V(C^c)$, $C^c$ is an induced subgraph of $G\setminus x_j$. This contradicts the fact that $C$ is a cycle of length greater than or equal to $5$.\par
  
	\noindent {\bf Case-II.}  Suppose $G$ does not contain any vertex other than the vertices of $C$, i.e., $G=C^c$. Since $I$ is componentwise linear, by \Cref{maximum}, we have $\mathrm{reg}(I)=d(I)$. If $\alpha_{i}(I)=w_{i}(I)$ for all $1\leq i\leq n$, then using \Cref{lemcochord}, one can derive that $\reg(I)>d(I)$. Thus, there exists $i\in [n]$ for which $\alpha_{i}(I)\neq w_{i}(I)$. In this case, note that $w_i(I)>1$. Let $V(G^c)=\{x_1,\ldots,x_n\}$ and $E(G^c)=\{\{x_{i},x_{i+1}\}\mid 1\leq i\leq n\,\,\text{and}\,\, x_{n+1}=x_{1}\}$. Without loss of generality, let $x_1$ be the vertex for which $\alpha_1(I)\neq w_1(I)$. Now, we have the following two possible subcases:\\
\noindent \textbf{Subcase-II(a).} For every $x_1^{w_1(I)}x_{k}^{b_k}\in \G(I)$, there exists $x_1^{a_1}x_k^{c_k}\in \G(I)$ for some $a_1<w_1(I)$ and $c_k>b_k$, where $x_k\in\N_{G}(x_1)=\{x_3,\ldots,x_{n-1}\}$. Since $I$ is componentwise linear, we have $I^{\PP}$ is componentwise linear by \Cref{compntpolariz}. Now, consider the square-free monomial ideal $J'$ such that $\G(J')=\{u\in \G(I^{\PP})\mid x_{1w_1(I)}\nmid u\}$. Then by \Cref{proplin}(2), $J'$ is componentwise linear. Now, due to our assumption, one can observe that $J'=I_{1}^{\PP}$, where $I_{1}$ is a support-two monomial ideal with the same underlying graph $G$ and $w_{1}(I_1)<w_1(I), w_{2}(I_1)=w_2(I),\ldots,w_{n}(I_1)=w_n(I)$. Therefore, we have $m(I_1)<m(I)$, which implies $I_{1}$ is not componentwise linear by the induction hypothesis as $G$ is not co-chordal. Thus, from \Cref{compntpolariz}, we can say that $J'$ is not componentwise linear, which gives a contradiction.\par

\noindent \textbf{Subcase-II(b).} Suppose there exists $x_{k}\in \N_{G}(x_1)$ such that $x_{1}^{w_1(I)}x_{k}^{b_k}\in \G(I)$ for some $b_k\geq 1$ and no other minimal generators of $I$ has support $\{x_1,x_k\}$. Let $A=\{x_{i_1},\ldots,x_{i_{r}}\}\subseteq \N_{G}(x_1)$ be such that $x_{1}^{w_1(I)}x_{i_j}^{b_{i_j}}\in \G(I)$ but no other minimal generators of $I$ has support $\{x_1,x_{i_j}\}$ for each $1\leq j\leq r$. Due to our choice of $x_1$ at the beginning (i.e., $\alpha_1(I)\neq w_1(I)$), we have $A\subsetneq \N_{G}(x_1)$. Now, consider the graph $G_1$ as follows: 
$$V(G_1)=V(G) \text{ and } E(G_1)=E(G)\setminus \{\{x_1,x_{i_1}\},\ldots,\{x_1,x_{i_r}\}\}.$$
Again, we recall the square-free monomial ideal $J'$ with $\G(J')=\{u\in \G(I^{\PP})\mid x_{1w_1(I)}\nmid u\}$. Since $I^{\PP}$ is componentwise linear by \Cref{compntpolariz}, it follows from \Cref{proplin}(2) that $J'$ is also componentwise linear. Next, let us define the monomial ideal $I_1$ such that 
$$\G(I_1)=\G(I)\setminus \{x_{1}^{w_1(I)}x_{i_j}^{b_{i_j}}\mid 1\leq j\leq r\}.$$ 
Then one can easily verify that $I_1$ is a support-two monomial ideal with the base graph $G_1$ and $I_1^{\PP}=J'$. If we look at the complement of $G_1$, then we can see that $E((G_1)^{c})=E(G^c)\cup \{\{x_1,x_{i_{1}}\},\ldots,\{x_{1},x_{i_{r}}\}\}$. Since $A\subsetneq \N_{G}(x_1)$, there should exists a vertex $x_{i}\in \N_{G}(x_1)$ such that $\{x_1,x_i\}\not\in E((G_1)^c)$. Therefore, $(G_1)^c$ contains an induced cycle of length greater than or equal to $4$ as $G^c$ is a cycle of length greater than or equal to $5$. Hence, $G_1$ is not co-chordal. From the construction of the ideal $I_1$, we see that $m(I_1)<m(I)$. Thus, by the induction hypothesis, $I_1$ is not componentwise linear. This gives a contradiction to the fact that $J'$ is componentwise linear due to \Cref{compntpolariz}.\par 
In each of the above cases, we arrive at a contradiction. Therefore, our initial assumption is false, and hence $G = G_I$ is co-chordal.
\end{proof}
The example below conveys that the converse of \Cref{thmcochordal} is not true.

 \begin{example}
Let $R = \mathbb{Q}[x_1,x_2,x_3]$ and $I=(x_1^3x_2^2, x_2^3x_3^4, x_3^4x_1^5).$ Then by Macaulay2 \cite{M2}, $\reg(I)=10$. Thus, by \Cref{maximum}, $I$ is not componentwise linear, whereas $G_I$ is co-chordal. 
\end{example}   

\section{Linear resolution of powers of minimal support-two monomial ideals}\label{powers-minimal-support}
In this section, we investigate the linearity of powers of support-two monomial ideals. In particular, for minimal support-two monomial ideals and their powers, we provide a complete characterization of those that admit a linear resolution. To make the proof of our main result more transparent, we first establish several lemmas, which in fact constitute essential parts of the proof.
\begin{proposition}\label{prop:reg ci}
 Let $I=(m_1,\ldots,m_r)$ be a complete intersection monomial ideal 
with $\deg(m_i)=d_i$ for all $1\leq i\leq r$. Then for all $k\geq 1$,
\[\displaystyle \reg(I^k)= (k-1) \max_{i\in [r]}\{d_i\}
+\sum_{i=1}^{r} d_i -r+1.\]
\end{proposition}
\begin{proof}
 The proof follows from \cite[Theorem 1.1]{nv19}.  
\end{proof}

\begin{lemma}\label{star}
Let $I=(x_1^{a_1}x_2^{a_2},x_2^{b_2}x_3^{b_3})$ be a support-two monomial ideal with $\alpha(I)>2$ and $G_{I}=P_3$. Then $I^k$ has a linear resolution if and only if $a_1=b_3=1$ and $a_2=b_2$. 
\end{lemma}
\begin{proof}
First, we assume $I^k$ has a linear resolution. Then $I$ is equigenerated, equivalently, $a_1 + a_2 =  b_2 + b_3=d$ (say). Consider the following cases: \par
\noindent{\bf Case-I.} Suppose $a_2 > b_2$. Then $a_2 + b_3 > b_2 + b_3 = d$ and
\(
I = (x_2^{\,b_2})\big(x_1^{a_1} x_2^{a_2 - b_2},\, x_3^{b_3}\big).
\)
This implies that $I^k=(x_2^{kb_2} \big)\big(x_1^{a_1} x_2^{a_2 - b_2},\, x_3^{b_3})^k$. Thus, by \Cref{prop:reg ci}, we have
\begin{align*}
\reg(I^k) &= kb_2 +(k-1)\max\{a_1+a_2-b_2,b_3\}+(a_1 + a_2 - b_2) + b_3 -2+ 1  \\
       &=kb_2 + (k-1)\{b_3\}+ a_1 + a_2 -b_2 + b_3 - 1 \\
       &= kd-b_3+d-b_2+b_3-1=kd+d-b_2+1.
\end{align*}
By the definition of linear resolution, we have $kd=\reg(I^k)=kd+d-b_2-1$. Thus, we have $b_2+b_3=d=b_2+1$, which implies that $b_3=1$. Also, we have $a_1+a_2=b_2+1$. Since $a_2 > b_2$, then the only choice is $a_1=0$ and $a_2=b_2+1$, which is not possible as $I$ is a support-two monomial ideal. \\
\noindent
{\bf Case-II.} Suppose $a_2 < b_2$. Then similarly, we get $a_1=1$, $b_3=0$ and $b_2=a_2+1$, which is not possible as $I$ is a support-two monomial ideal. \\ 
\noindent
{\bf Case-III.} Suppose $a_2 = b_2=t$. Then $I=(x_2^t)(x_1^{a_1},x_3^{b_3})$. This implies that $I^k=(x_2^{kt})(x_1^{a_1},x_3^{b_3})^k$. Thus, using \Cref{prop:reg ci}, we obtain 
$$\reg(I^k) = tk + (k-1)\max\{a_1,b_3\}+a_1 + b_3 - 1.$$ 
Again, using $a_1+a_2=d=b_2+b_3$ and  $a_2 = b_2=t$, we get $a_1=b_3=d-t$. Since $I^k$ is linear, $\reg(I^k) = kd= tk +(k-1)(d-t)+2(d-t)- 1$. This gives $d-t=1$. Therefore, $a_1+a_2=d$ and $a_2=t$ together imply $a_1=1$. Similarly, we have $b_3=1$.\par

\noindent Considering all the above cases, we see that $a_1 = b_3=1$ and $a_2 = b_2$.\par 

Conversely, assume $a_1=b_3=1$ and $a_2=b_2$. Then $I$ is equigenerated in degree $d=1+a_2=1+b_2$. Thus, using \Cref{prop:reg ci}, we obtain 
\begin{align*}
    \reg(I^k)&=\reg(x_2^{k(d-1)}(x_1,x_3)^k)\\
    &= k(d-1)+(k-1)+2-2+1\\
    &=kd.
\end{align*}
\noindent Hence, the ideal $I^k$ has a linear resolution.
\end{proof}
\begin{lemma}\label{path}
Let $I=(x_1^{a_1}x_2^{a_2},x_2^{b_2}x_3^{b_3},x_3^{c_3}x_4^{c_4})$ be a support-two monomial ideal with the base graph $P_4$. If $\alpha(I)>2$, then $I^k$ has no linear resolution for all $k \geq 1$. 
\end{lemma}
\begin{proof}
If possible, suppose $I^k$ has a linear resolution for some $k\geq 1$. Then by \Cref{proplin}(1), the ideals $(x_1^{a_1}x_2^{a_2},x_2^{b_2}x_3^{b_3})^k$ and $(x_2^{b_2}x_3^{b_3},x_3^{c_3}x_4^{c_4})^k$ have linear resolutions. Now, applying \Cref{star} to these two ideals, we get $a_1=1$, $a_2=b_2$, $b_3=1$ and $b_2=1,b_3=c_3, c_4=1$. Consequently, we obtain $I=(x_1x_2,x_2x_3,x_3x_4)$, which contradicts the fact that $\alpha(I)>2$. Hence, $I^k$ has no linear resolution for every $k\geq 1$.
\end{proof}

\begin{proposition}\label{lem:triangle}
Let $I=(x_1^{a_1}x_2^{a_2},x_2^{b_2}x_3^{b_3},x_1^{c_1}x_3^{c_3})$ be a support-two monomial ideal with the base graph $C_3$. If $\alpha(I)>2$, then $I$ has no linear resolution. 
\end{proposition}
\begin{proof}
Suppose $I$ has a linear resolution. Then $I$ is equigenerated, and thus, we have $a_1+a_2=b_2+b_3=c_1+c_3=d$ (say). Let us consider the following two possible cases: \\
\noindent {\bf Case-I.} Suppose that for each $1 \le i \le 3$, the exponent of the variable $x_i$ is the same in every generator of $I$ in which it appears; that is, $a_1 = c_1$, $a_2 = b_2$, and $b_3 = c_3$. Then we have $a_1=c_1=a_2=b_2=b_3=c_3=t$ (say) as $I$ is equigenerated. This implies that  $I=(x_1^{t}x_2^{t},x_2^{t}x_3^{t},x_1^{t}x_3^{t})$. Now, we write $I=J+K$, where $J=(x_1^{t}x_2^{t},x_2^{t}x_3^{t})$ and $K=(x_1^{t}x_3^{t})$. Note that $J \cap K=(x_1x_2x_3)^t$, and therefore, $\reg(J \cap K)=3t$ by \Cref{prop:reg ci}. Again, writing $J=x_{2}^{t}(x_{1}^t,x_{3}^t)$ and using \Cref{prop:reg ci}, one can deduce that $\reg(J)=3t-1$ and $\reg(K)=2t$. Now, consider the following exact sequence:
\begin{equation}\label{eq1}
    0 \rightarrow J \cap K \rightarrow J \oplus K \rightarrow J+K \rightarrow 0.
\end{equation}
Observe that $3t=\reg(J\cap K) > \reg(J \oplus K)=\max\{\reg(J),\reg(K)\}=3t-1$. Thus, by applying \Cref{lem:regularity-lemma}(2) \& (3) on the exact sequence \eqref{eq1}, we have $\reg(I)=\reg(J \cap K)-1=3t-1$. Since $I$ has a linear resolution, we should have $\reg(I)=3t-1=d=2t$. This gives us $t=1$, which is a contradiction to the fact that $\alpha(I)>2$.\\
\noindent{\bf Case-II.} Suppose there exists a vertex, say $x_1$, whose exponents are different in the generators of $I$, i.e., $a_1\neq c_1$. Without loss of generality, we assume $a_1<c_1$. Since we assumed $I$ has a linear resolution, $I^{\PP}$ has a linear resolution by \Cref{compntpolariz}. Now, observe that $$\left(I^{\PP},x_{1c_1}\right)=\left(\prod_{i=1}^{a_1}x_{1i}\prod_{i=1}^{a_2}x_{2i},\prod_{i=1}^{b_2}x_{2i}\prod_{i=1}^{b_3}x_{2i}, x_{1c_1}\right)=(J^{\PP},x_{1c_1}),$$ 
where $J=(x_1^{a_1}x_2^{a_2},x_2^{b_2}x_3^{b_3})$. Then it follows from \Cref{proplin}(1) that $J^{\PP}$ has a linear resolution, and consequently, $J$ has a linear resolution by \Cref{compntpolariz}. Therefore, due to \Cref{star}, we have $a_1=b_3=1$ and $a_2=b_2=d-1$, i.e., $I=(x_1x_2^{d-1},x_2^{d-1}x_3,x_1^{c_1}x_3^{c_3})$. Now, let us write $I=J+K$, where $J=(x_1x_2^{d-1},x_2^{d-1}x_3)$ and $K=(x_1^{c_1}x_3^{c_3})$. Note that $J \cap K=(x_1^{c_1}x_2^{d-1}x_3^{c_3})$. Thus, using \Cref{prop:reg ci}, one can easily obtain that $\reg(J \cap K)=c_1+c_2+d-1=2d-1$, $\reg(J)=d$, and $\reg(K)=d$. In this case, we again have $\reg(J \cap K) > \reg(J \oplus K)=\max\{\reg(J),\reg(K)\}$. Thus, by applying regularity lemma \cite[Corollary 20.19]{e95} on exact sequence \eqref{eq1}, we get $\reg(I)=\reg(J \cap K)-1=2d-2$. Since $I$ has a linear resolution, we should have $d=2$, which contradicts the fact that $\alpha(I)>2$.\par 
In both cases, we arrive at a contradiction to the assumption that $I$ has a linear resolution. Hence, $I$ has no linear resolution.
\end{proof}

\begin{lemma}\label{powertriangle}
Let $I=(x_1^{a_1}x_2^{a_2},x_2^{b_2}x_3^{b_3},x_1^{c_1}x_3^{c_3})$ be a support-two monomial ideal with the base graph $C_3$. If $\alpha(I)>2$, then $I^k$ does not have a linear resolution for all $k\geq 1$.
\end{lemma}
\begin{proof}
We already proved in \Cref{lem:triangle} that $I$ has no linear resolution. If $I^k$ is not equigenerated, then $I^k$ can not have a linear resolution. Therefore, we may assume $I^k$ is equigenerated, equivalently, $I$ is equigenerated. This implies that $a_1+a_2=b_2+b_3=c_1+c_3=d$ (say). We consider two possible cases: \\
{\bf Case-I.} Suppose that for each $1 \le i \le 3$, the exponent of the variable $x_i$ is the same in every generator of $I$ in which it appears; that is, $a_1 = c_1$, $a_2 = b_2$, and $b_3 = c_3$. Then we have $a_1=c_1=a_2=b_2=b_3=c_3=t$ (say) as $I$ is equigenerated. This implies that  $I=(x_1^{t}x_2^{t},x_2^{t}x_3^{t},x_1^{t}x_3^{t})$. We proceed by induction on $k$ to show that $I^k$ does not have a linear resolution. Now, we will consider three subcases:\\
{\bf Subcase-I(a).} Let $k=2$. Then one can observe that
$((I^2)^{\PP},x_{1(2t)})=(((x_2^tx_3^t)I)^{\PP}, x_{1(2t)})$. Note that by \Cref{lem:triangle}, $I$ has no linear resolution as $\alpha(I)=d>2$. Thus, we have $\reg(I)>d= 2t$, which implies $\reg((x_2^tx_3^t)I)=2t+\reg(I) > 4t$. Thus, $(x_2^tx_3^t)I$ does not have a linear resolution, and hence, by \Cref{compntpolariz}, $((x_2^tx_3^t)I)^{\PP}$ has no linear resolution. Therefore, due to \Cref{proplin}(1), $(I^2)^{\PP}$ has no linear resolution. Consequently, $I^2$ has no linear resolution by \Cref{compntpolariz}.
\par 

\noindent{\bf Subcase-I(b).} Suppose $k=3$. In this case, we see that
$$((I^3)^{\PP},x_{1(3t)},x_{2(3t)})=(((x_2^tx_3^t)(x_1^tx_3^t)I^2)^{\PP}, x_{1(3t)},x_{2(3t)}).$$ 
From Subcase-I(a), we see that $I^2$ has no linear resolution. Then we have $\reg(I^2)> 4t$, and so, $\reg((x_2^tx_3^t)(x_1^tx_3^t)I^2) > 8t$. Thus, by the definition of linearity and \Cref{compntpolariz}, $((x_2^tx_3^t)(x_1^tx_3^t)I^2)^{\PP}$ has no linear resolution. Therefore, $(I^2)^{\PP}$ has no linear resolution due to \Cref{proplin}(1). Hence, using \Cref{compntpolariz} again, we see that $I^2$ is not linear. \par 

\noindent {\bf Subcase-I(c).} Now, let us take $k>3$. Recall that $I^k=(x_1^tx_2^t,x_2^tx_3^t,x_1^tx_3^t)^k$. In this case, one can easily prove that
$$\left(\left(I^k\right)^{\PP},x_{1(kt)},x_{2(kt)},x_{3(kt)}\right)=\left(\left((x_1^tx_2^t)(x_2^tx_3^t)(x_1^tx_3^t)I^{k-3}\right)^{\PP},x_{1(kt)},x_{2(kt)},x_{3(kt)}\right).$$
By the induction hypothesis, $I^{k-3}$ has no linear resolution. Therefore, we have $\reg(I^{k-3})> 2t(k-3)$, which gives $\reg\left((x_1^tx_2^t)(x_2^tx_3^t)(x_3^tx_1^t)I^{k-3}\right)> 2t(k-3)+6t$. This fact together with \Cref{compntpolariz} imply that $\left((x_1^tx_2^t)(x_2^tx_3^t)(x_3^tx_1^t)I^{k-3}\right)^{\PP}$ has no linear resolution. Hence, by \Cref{proplin}(1), $(I^k)^{\PP}$ has no linear resolution. Consequently, again by \Cref{compntpolariz}, it follows that $I^k$ has no linear resolution.
\\
{\bf Case-II.} Suppose there exists a vertex, say $x_1$, whose exponents are different in the generators of $I$, i.e., $a_1\neq c_1$. Without loss of generality, we assume $a_1<c_1$. We will prove, by induction on $k$, that $I^k$ does not have a linear resolution. The base case $k=1$ follows from \Cref{lem:triangle}. Assume that $I^{k-1}$ has no linear resolution. Now, let us write
$$I^k=(x_1^{a_1}x_2^{a_2},x_2^{b_2}x_3^{b_3},x_1^{c_1}x_3^{c_3})^k=J+K,$$
where $J=(x_1^{a_1}x_2^{a_2},x_2^{b_2}x_3^{b_3})I^{k-1}$ and $K=(x_1^{c_1}x_3^{c_3})^k$. Next, splitting $J=x_{1}^{a_1}x_{2}^{a_2}I^{k-1}+x_{2}^{b_2}x_{3}^{b_3}I^{k-1}$, we consider the following exact sequence
\begin{equation}\label{eq2}
    0 \rightarrow x_1^{a_1}x_2^{\max\{a_2,b_2\}}x_3^{b_3}I^{k-1} \rightarrow x_1^{a_1}x_2^{a_2}I^{k-1} \oplus x_2^{b_2}x_3^{b_3}I^{k-1}  \rightarrow J \rightarrow 0.
\end{equation}
Let $A=x_1^{a_1}x_2^{\max\{a_2,b_2\}}x_3^{b_3}I^{k-1}$ and $B=x_1^{a_1}x_2^{a_2}I^{k-1} \oplus x_2^{b_2}x_3^{b_3}I^{k-1}$. Then, we have 
\begin{align*}
    \reg(A)&=\reg(I^{k-1})+a_1+b_3+\max\{a_2,b_2\}, \text{ and}\\
    \reg(B)&=\max\{\reg(I^{k-1})+a_1+a_2, \reg(I^{k-1})+b_2+b_3\}.
\end{align*} 
Note that $\reg(A)>\reg(B)$. Thus, by applying \Cref{lem:regularity-lemma}(2) \& (3) to the short exact sequence \eqref{eq2}, we get $\reg(J) = \reg (A)-1$. Therefore, we have $\reg(J)> (k-1)d+d=kd$, because $a_1+a_2=b_2+b_3=d$ and by induction hypothesis, $\reg(I^{k-1})>(k-1)d$. Thus, $J$ has no linear resolution, and consequently, $J^{\PP}$ has no linear resolution by \Cref{compntpolariz}. Now, observe that $((I^k)^{\PP},x_{1(c_1k)})=(J^{\PP},x_{1(c_1k)})$. Therefore, by \Cref{proplin}(1), $(I^{k})^{\PP}$ does not have a linear resolution. Hence, by \Cref{compntpolariz}, $I^k$ has no linear resolution.
\end{proof}
\begin{lemma}\label{lem:fourcycle}
Let $I=(x_1^{a_1}x_2^{a_2},x_2^{b_2}x_3^{b_3},x_3^{c_3}x_4^{c_4},x_4^{d_4}x_1^{d_1})$ be a support-two monomial ideal with the base graph $C_4$. If $\alpha(I)>2$, then $I^k$ does not have a linear resolution for all $k\geq 1$.    
\end{lemma}
\begin{proof}
  If possible, suppose $I^k$ has a linear resolution. Then by \Cref{proplin}(1), the ideals $(x_1^{a_1}x_2^{a_2},x_2^{b_2}x_3^{b_3})^k$,  $(x_2^{b_2}x_3^{b_3},x_3^{c_3}x_4^{c_4})^k$ and $(x_3^{c_3}x_4^{c_4},x_4^{d_4}x_1^{d_1})^k$ have linear resolutions. Thus, we have $a_1=1$, $a_2=b_2$, $b_3=1$ and $b_2=1,b_3=c_3, c_4=1$ and $c_3=1, c_4=d_4, d_1=1$ due to \Cref{star}. This implies that $I=(x_1x_2,x_2x_3,x_3x_4)$, which contradicts the fact that $\alpha(I)>2$. Hence, $I^k$ has no linear resolution for all $k \geq 1$.    
\end{proof}
\begin{theorem}\label{equivalent}
Let $I\subseteq R$ be a minimal support-two monomial ideal with $\alpha(I) > 2$. Then the following conditions are equivalent:
\begin{enumerate}
    \item $I$ has a linear resolution.
    \item $I=(x_i x_j^{a_j} \mid i\in [n]\setminus \{j\})$ for some $1 \le j \le n$ and $a_j>1$.
    \item $I^k$ has a linear resolution for all $k \ge 1$.
    \item $I^k$ has a linear resolution for some $k \geq 1$.
    \item $I^k$ has linear quotients for all $k\geq 1$.
\end{enumerate}    
\end{theorem}
\begin{proof}
$(1)\implies (2)$: Suppose that $I$ has a linear resolution. Let the base graph of $I$ be $G_I$, i.e., $\sqrt{I}=I(G_I)$. Then, by \Cref{proplin}(1), $I_{[H]}$ has a linear resolution, where $H$ is an induced subgraph of $G_I$ and $I_{[H]}=\left(f\in\G(I)\mid \supp(f)\in E(H)\right)$. Since $I$ is equigenerated and $\alpha(I)>2$, we have $\alpha(I_{[H]})>2$. Therefore, the graph $H$ can not be isomorphic to $P_{4}$, $C_3$, or $C_4$ due to \Cref{path}, \Cref{lem:triangle} and \ref{lem:fourcycle}, respectively. This is true for any induced subgraph of $G$. Hence, $G_I$ must be a star graph. Moreover, by \Cref{proplin}(1) and \ref{star}, the only possible form of $I$ to have a linear resolution is $I = (x_1 x_j^{a_j},\ldots, x_{j-1}x_j^{a_j},x_{j+1}x_j^{a_j}, \ldots, x_n x_j^{a_j})$ for some $j\in\{1,\ldots,n\}$ and $a_j>1$. \par
$(2)\implies (3)$: Without loss of generality, we assume $j=1$. Then, we can write $I = x_1^{a_1}(x_2,\ldots,x_n)$. Thus, for any $k\geq 1$, it follows from \Cref{prop:reg ci} that
$$\reg(I^k)=\reg(x_1^{ka_1}(x_2,\ldots,x_n)^k)=k(a_1+1)=\alpha(I^k).$$ 
Hence, by definition, $I^k$ has a linear resolution for all $k \geq 1$. \par 
$(3)\implies (4)$: This case is obvious. \par 
$(4)\implies (2)$: Given that $I^k$ has a linear resolution for some $k\geq 1$. Let $I'=(u\in \G(I)\mid x\nmid u)$. Then observe that $(I')^k=(u\in \G(I^{k})\mid x\nmid u)$. Then, by \Cref{proplin}(1), $(I')^k$ also has a linear resolution. Therefore, if the support-two monomial ideal $I$ has the base graph $G_I$, then $(I_{[H]})^k$ has a linear resolution, where $H$ is an induced subgraph of $G_I$ and $I_{[H]}=\left(f\in\G(I)\mid \supp(f)\in E(H)\right)$. Since $\alpha(I)>2$, we have $\alpha(I_{[H]})>2$ as $I$ is equigenerated. Thus, the induced subgraph $H$ can not be isomorphic to $P_4$, $C_3$, or $C_4$ by \Cref{path}, \ref{powertriangle} and \ref{lem:fourcycle}, respectively. Hence, the only possibility is that $G_I$ is a star graph as $G_I$ is co-chordal by \Cref{thmcochordal}. Finally, from \Cref{proplin}(1) and \ref{star}, one can obtain that $I$ is of the form $(x_1 x_j^{a_j},\ldots, x_{j-1}x_j^{a_j},x_{j+1}x_j^{a_j}, \ldots, x_n x_j^{a_j})$ for some $j\in\{1,\ldots,n\}$ and $a_j>1$. \par

$(2)\implies (5)$: By the condition of $(2)$, we have $I = (x_{j}^{a_j})(x_1, \ldots, x_{i-1}, x_{i+1}, \ldots, x_n)$ for some $1\leq j\leq n$ and $a_j>1$. Note that $(x_j^{a_j})$ and $(x_1, \ldots, x_{i-1}, x_{i+1}, \ldots, x_n)$ are polymatroidal ideals (for the definition of polymatroidal ideals, see \cite[Definition 12.6.1]{hh11}). Then thanks to \cite[Theorem 12.6.3]{hh11}, $I^k$ has linear quotients for all $k \ge 1$. \par

$(5)\implies (1)$: Given that $I^k$ has linear quotients for all $k\geq 1$. Thus, $I$ has linear quotients, and to show $I$ has a linear resolution, it is enough to show that $I$ is equigenerated. Let $I$ has linear quotient property with respect to the order $f_1<\cdots<f_r$, where $\G(I)=\{f_1,\ldots,f_r\}$. Suppose $I$ is not equigenerated and  
\[
\deg(f_1) = \cdots = \deg(f_{i-1}) = d \ne \deg(f_i)
\]
for some $2 \le i \le r$. Write $f_l = x_u^{u_1} x_v^{v_1}$ for some $1 \le l \le i-1$, and $f_i = x_a^{a_1} x_b^{b_1}$, where $b \ne v$, $u_1+v_1=d$ and $a_1+b_1 \neq d$. If $a = u$ with $u_1\leq a_1$ or $b = u$ with $u_1 \leq b_1$, then $(f_l : f_i) = (x_v^{v_1})$. If $a = u$ with $u_1 > a_1$, then $(f_l : f_i) = (x_u^{u_1-a_1}x_v^{v_1})$. If $b = u$ with $u_1 > b_1$, then $(f_l : f_i) = (x_u^{u_1-b_1}x_v^{v_1})$.
If $a = v$ with $v_1 \leq a_1$, then $(f_l : f_i) = (x_u^{u_1}x_v^{a_1-v_1})$. If $a = v$ with $a_1 < v_1$, then $(f_l : f_i) = (x_u^{u_1}x_v^{v_1-a_1})$. If $a, b, u, v$ are distinct, then 
$(f_l : f_i) = (x_u^{u_1} x_v^{v_1})$. In all cases, one can check that $x_v \notin (f_j : f_i)$ and $x_u \notin (f_j : f_i)$ for all 
$1 \le j \le i - 1$. Thus, the ideal $((f_1,\ldots,f_{i-1}):f_i)$ can not be generated by a set of variables, which contradicts the assumption that $I$ has linear quotients with respect to the order $f_1<\cdots<f_r$. Therefore, $I$ is equigenerated, and hence, by \cite[Proposition 8.2.1]{hh11}, $I$ has a linear resolution.
\end{proof}

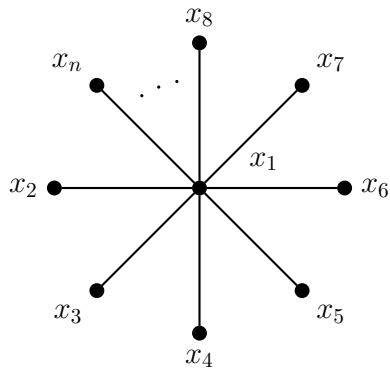
\begin{figure}[ht]
\centering
\begin{tikzpicture}[
    scale=0.80,
    transform shape,
    vertex/.style={
        circle,
        draw,
        fill=black,
        minimum size=6pt,
        inner sep=0pt
    },
    dot/.style={
        circle,
        fill=black,
        minimum size=1.5pt,
        inner sep=0pt
    },
    every node/.style={font=\large},
    thick
]

\def\r{2.4}

% Center vertex
\node[vertex] (v1) at (0,0) {};

% Leaves
\node[vertex,label=left:$x_2$]        (v2) at (180:\r) {};
\node[vertex,label=below left:$x_3$]  (v3) at (225:\r) {};
\node[vertex,label=below:$x_4$]       (v4) at (270:\r) {};
\node[vertex,label=below right:$x_5$] (v5) at (315:\r) {};
\node[vertex,label=right:$x_6$]       (v6) at (360:\r) {};
\node[vertex,label=above right:$x_7$] (v7) at (45:\r) {};
\node[vertex,label=above:$x_8$]       (v8) at (90:\r) {};
\node[vertex,label=above left:$x_n$]  (vn) at (135:\r) {};

% Edges
\foreach \x in {v2,v3,v4,v5,v6,v7,v8,vn}
    \draw (v1)--(\x);

% Center label
\node[fill=white,inner sep=1pt] at ($(v1)+(22.5:1.15)$) {$x_1$};

% Dots
\foreach \a in {102,112,122}
    \node[dot] at ($(v1)+(\a:1.8)$) {};

\end{tikzpicture}

\caption{The base graph of a linear minimal support-two monomial ideal}
\label{fig:star}
\end{figure}

\begin{remark}\rm{
The above theorem includes, as special cases, the edge ideals of weighted oriented graphs and vertex-weighted graphs that admit a linear resolution. It also shows that, in the case of edge ideals of edge-weighted graphs, a linear resolution cannot occur when the minimum degree of the generators exceeds $2$.
}
\end{remark}

The following examples show that, for a support-two monomial ideal $I$, if we drop the assumption $\mu(I)=\mu(\sqrt{I})$, then the necessary and sufficient condition given in \Cref{equivalent} for linearity is no longer necessary.

\begin{example}\label{ex:lin1}
Consider the monomial ideal $I=(x_1x_2^2,x_2x_3^2,x_2^2x_3,x_3^2x_4) \subseteq \mathbb{Q}[x_1,x_2,x_3,x_4]$. Then $I$ is a support-two monomial ideal which is not minimal, and the base graph $G_I$ is a path of length three. Using Macaulay2 \cite{M2}, we have $\reg(I)=3$. Thus, $I$ has a linear resolution.
\end{example}
\begin{example}\label{ex:lin2}
Consider the monomial ideal $I=(x_1x_2^2,x_2x_3^2,x_2^2x_3,x_3^2x_1) \subseteq \mathbb{Q}[x_1,x_2,x_3]$. Then $I$ is a support-two monomial ideal which is not minimal, and the base graph $G_I$ is a triangle. By Macaulay2 \cite{M2}, we get $\reg(I)=3$. Thus, $I$ has a linear resolution.    \end{example}

\section{Classification of linear support-two monomial ideals}\label{sec5}
In this section, we provide a complete classification of linear support-two monomial ideals. In view of Fr'oberg's theorem and \Cref{equivalent}, it remains to consider only non-minimal support-two monomial ideals (i.e., support-two monomial ideal $I$ with $\mu(I)>\mu(\sqrt{I})$). Since one side of the proof is lengthy and relies on the analysis of forbidden structures, we divide it into several lemmas for the convenience of the reader. We then combine these intermediate results to establish the main theorem. 

\begin{figure}[htbp]
\begin{center}
\begin{tikzpicture}[
    v/.style={
        circle,
        draw,
        fill=black,
        inner sep=0pt,
        minimum size=6pt
    },
    lbl/.style={font=\large},
    red edge/.style={red, very thick}
]

\begin{scope}[shift={(0,0)}]
    \node[v, label=above:$x_1$] (a1) at (0,0) {};
    \node[v, label=above:$x_2$] (a2) at (1.5,0) {};
    \draw[red edge] (a1)--(a2);
    \node[lbl] at (0.75,-1) {$P_2$};
\end{scope}

% ---------- P_3 ----------
\begin{scope}[shift={(6,0)}]
    \node[v, label=above:$x_1$] (a1) at (0,0) {};
    \node[v, label=above:$x_2$] (a2) at (1.5,0) {};
    \node[v, label=above:$x_3$] (a3) at (3,0) {};
    \draw[red edge] (a1)--(a2);
    \draw (a2)--(a3);
    \node[lbl] at (1.5,-1) {$P_3$};
\end{scope}
% ---------- P_4 ----------
\begin{scope}[shift={(-1.5,-4)}]
    \node[v, label=above:$x_1$] (b1) at (0,0) {};
    \node[v, label=above:$x_2$] (b2) at (1.5,0) {};
    \node[v, label=above:$x_3$] (b3) at (3,0) {};
    \node[v, label=above:$x_4$] (b4) at (4.5,0) {};
    \draw[red edge] (b1)--(b2);
    \draw (b2)--(b3)--(b4);
    \node[lbl] at (2.25,-1) {$P_4$};
\end{scope}
% ---------- C_3 ----------
\begin{scope}[shift={(7.5,-4)}]
    \node[v, label=above:$x_1$]       (c1) at (90:1.2)  {};
    \node[v, label=below left:$x_2$]  (c2) at (210:1.2) {};
    \node[v, label=below right:$x_3$] (c3) at (330:1.2) {};
    \draw (c2)--(c3)--(c1);
    \draw[red edge] (c1)--(c2);
    \node[lbl] at (0,-1.5) {$C_3$};
\end{scope}
% ---------- C_3 with one whisker (leaf x_1 attached to x_2) ----------
\begin{scope}[shift={(4,-9)}]
    \node[v, label=right:$x_2$]       (e2) at (90:1.2)  {};
    \node[v, label=below left:$x_3$]  (e3) at (210:1.2) {};
    \node[v, label=below right:$x_4$] (e4) at (330:1.2) {};
    \node[v, label=right:$x_1$]       (e1) at (90:2.7)  {};
    \draw (e2)--(e3)--(e4)--(e2);
    \draw[red edge] (e1)--(e2);
    \node[lbl] at (0,-1.5) {$C_3$ with a whisker};
\end{scope}
\end{tikzpicture}
\end{center}
\caption{Base graphs of \Cref{lem:k2} to \ref{prop:k3}}
\label{fig-1}
\end{figure}
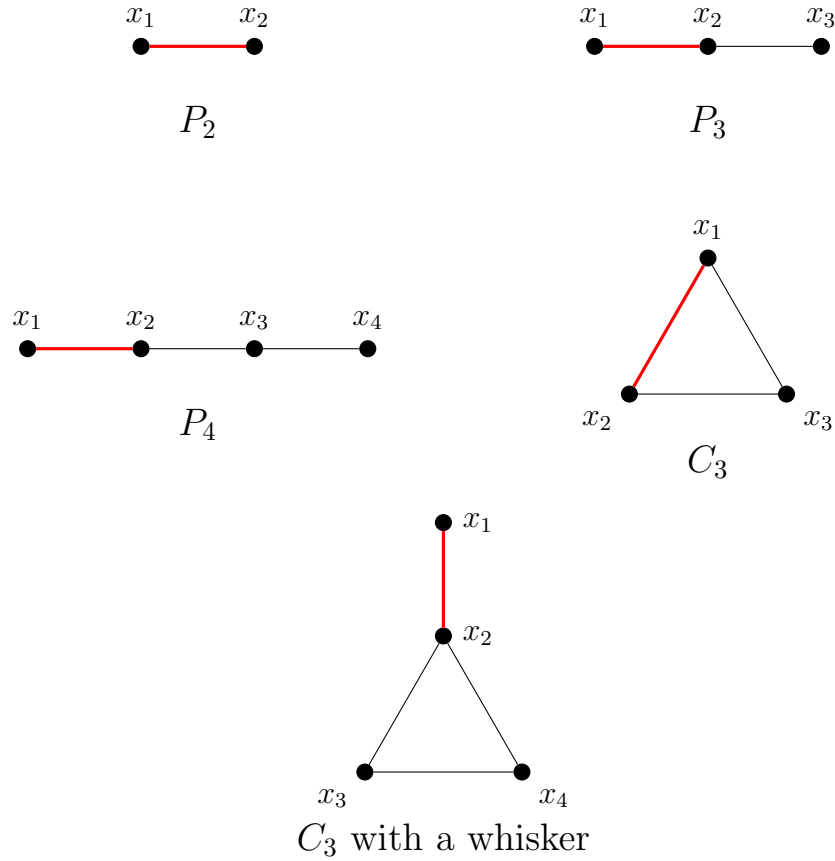

\begin{lemma}\label{lem:k2}
 Let $I=(x_1^{a_1}x_2^{b_1}, x_1^{a_2}x_2^{b_2}, \ldots, x_1^{a_k}x_2^{b_k}) \subseteq \mathbb{K}[x_1,x_2]$ be any support-two monomial ideal with the base graph $K_2$, where $a_1 > a_2 > \ldots > a_k >0$ and $0 < b_1 < b_2 < \ldots < b_{k - 1} < b_k$. Then $I$ has linear resolution if and only if $a_{i+1}=a_i-1$ for all $1 \leq i \leq k-1$ and $b_{i+1}=b_i+1$ for all $1 \leq i \leq k-1$.  
\end{lemma}
\begin{proof}
Note that $a_i+b_i=d$ for all $1 \leq i \leq k$. Suppose $a_{i+1}=a_i-1$ for all $1 \leq i \leq k-1$ and $b_{i+1}=b_i+1$ for all $1 \leq i \leq k-1$. Then $b_{i+1}=d-a_i+1$. Thus, we have $a_{i+1}=a_1-i$ for all $1 \leq i \leq k-1$. Therefore, we have 
\begin{align*}
    I=&(x_1^{a_1}x_2^{d-a_1}, x_1^{a_2}x_2^{d-a_2}, \ldots, x_1^{a_k}x_2^{d-a_{k}}) \\
    =& (x_1^{a_1}x_2^{d-a_1}, x_1^{a_1-1}x_2^{d-a_1+1}, \ldots, x_1^{a_1-(k-1)}x_2^{d-a_{1}+(k-1)}) \\
    =& x_1^{a_1-(k-1)}x_2^{d-a_1}(x_1^{k-1}, x_1^{k-2}x_2,\ldots, x_2^{k-1}) \\
    =& x_1^{a_1-(k-1)}x_2^{d-a_1} (x_1,x_2)^{k-1}
\end{align*}
Therefore by \Cref{prop:reg ci}, $\reg(I)=a_1-(k-1)+d-a_1+(k-2)+1=d$. This implies that $I$ has linear resolution. \\
Conversely, suppose $I$ has linear resolution. Then $\reg(I)=d$. If possible, suppose $a_{i+1} \leq a_i-2$. Consider the monomial $m=x_1^{a_{i+1}+1}x_2^{d-(a_i-1)}.$
If $m\in I$, then there exists $j$ such that $x_1^{a_j}x_2^{d-a_j}\mid m.$ Hence, $a_j\leq a_{i+1}+1$ and $a_j\geq a_i-1.$
Therefore, $a_i-1\leq a_j\leq a_{i+1}+1.$
Since $a_{i+1}\leq a_i-2$, we have $a_{i+1}+1\leq a_i-1,$
forcing $a_j=a_i-1$. This is impossible because there is no exponent
$a_j$ between $a_i$ and $a_{i+1}$. Hence
$x_1^{a_{i+1}+1}x_2^{d-(a_i-1)}\notin I.$ Then $x_1^{a_{i+1}+1}x_2^{d-(a_i-1)} \not \in I$. Thus, we have
$\reg(I) \geq \max\{d+a_{i+1}-a_i-1 \mid i=0,\ldots, k-1\}.$ This implies that $\reg(I) > d$ as $a_{i+1} \geq a_i+2$. This is a contradiction, as $\reg(I)=d$. Thus, $a_{i+1}=a_i-1$ and $b_{i+1}=d-a_{i+1}=d-a_i+1=b_i+1$.
\end{proof}

\begin{definition}{\rm
  Let $I=(x_1^{a_1}x_2^{b_1}, x_1^{a_2}x_2^{b_2}, \ldots, x_1^{a_k}x_2^{b_k}) \subseteq \mathbb{K}[x_1,x_2]$ be a monomial ideal with a base graph $K_2$, where $a_1 > a_2 > \ldots > a_k >0$ and $0 < b_1 < b_2 < \ldots < b_{k - 1} < b_k$. Then $I$ is said to be {\it tight} if $a_{i+1}=a_i-1$ for all $1 \leq i \leq k-1$ and $b_{i+1}=b_i+1$ for all $1 \leq i \leq k-1$. For any support-two monomial ideal $I$ with the base graph $G$, we call it tight if for each $e\in E(G)$, the ideal $I_{e}:=(u\in \G(I)\mid \supp(u)=e)$ is tight.
  }
\end{definition}

\begin{remark}{\rm
    In \Cref{fig-1}, the red coloured edge $\{x_1,x_2\}$ indicates that the deal $I$ has more than one minimal generators with support $\{x_1,x_2\}$. We use the same convention in \Cref{fig-2} and \ref{fig-3}.
    }
\end{remark}

\begin{lemma}\label{lem:p3}
 Let $I=(x_1^{a_1}x_2^{b_1},x_1^{a_1-1}x_2^{b_1+1},\ldots, x_1^{a_1-k}x_2^{b_1+k}, x_2^{c_1}x_3^{d_1}, x_2^{c_1-1}x_3^{d_1+1},\ldots, x_2^{c_1-l}x_3^{d_1+l})$ be a support-two monomial ideal in $R=\mathbb{K}[x_1,x_2,x_3]$ with the base graph $P_3$ and $k \geq 1$. If $I$ has a linear resolution, then $l=0$, $c_1=b_1+k$, $d_1=1$ and $a_1-k=1$, i.e. $I=(x_1^{k+1}x_2^{b_1},\ldots,x_1x_2^{b_1+k},x_2^{b_1+k}x_3)$.
\end{lemma}
\begin{proof}
 Suppose that $I$ has a linear resolution. Observe that 
 $$(I^{\PP},x_{2 (c_1-l+1)},x_{1 a_1-(k-1)})=
\big((x_1^{a_1-k}x_2^{b_1+k},\, x_2^{c_1-l}x_3^{d_1+l})^{\PP}, x_{2 (c_1-l+1)},x_{1 a_1-(k-1)}\big)$$
and 
$$(I^{\PP},x_{3 d_1+1},x_{1 a_1-(k-1)})
=\big((x_1^{a_1-k}x_2^{b_1+k},\, x_2^{c_1}x_3^{d_1})^{\PP},x_{3 d_1+1},x_{1 a_1-(k-1)}\big).$$ 
Then, due to \Cref{compntpolariz} and \Cref{proplin}, the ideals $(x_1^{a_1-k}x_2^{b_1+k},\, x_2^{c_1-l}x_3^{d_1+l})$ and $(x_1^{a_1-k}x_2^{b_1+k},\, x_2^{c_1}x_3^{d_1})$ have linear resolutions. Hence by \Cref{equivalent}, we have
\[
a_1-k=1,\qquad b_1+k=c_1-l,\qquad d_1+l=1, b_1+k=c_1,\qquad d_1=1.
\]
Consequently, $l=0$, $c_1=b_1+k$, $d_1=1$ and $a_1-k=1$.
\end{proof}

\begin{lemma}\label{lem:p4}
 Let 
 $$I=(x_1^{a_1}x_2^{b_1}, \ldots, x_1^{a_1-k}x_2^{b_1+k}, x_2^{c_1}x_3^{d_1},\ldots, x_2^{c_1-l}x_3^{d_1+l}, x_3^{e_1}x_4^{f_1},\ldots, x_3^{e_1-t}x_4^{f_1+t})$$ 
 be a support-two monomial ideal with the base graph $P_4$. If $k \geq 1$, then $I$ has no linear resolution.   
\end{lemma}
\begin{proof}
  If possible, suppose that $I$ has a linear resolution. Let us consider the ideal $I'=(u\in \G(I)\mid x_4\nmid u)$. Since $I$ has a linear resolution, $I'$ also has a linear resolution by \Cref{proplin}. Thus, applying \Cref{lem:p3} to the ideal $I'$, we get 
  $$I=(x_1^{k+1}x_2^{b_1},\ldots,x_1x_2^{b_1+k},x_2^{b_1+k}x_3, x_3^{e_1}x_4^{f_1},\ldots, x_3^{e_1-t}x_4^{f_1+t}).$$ 
  Now, one can observe that
  \[
  (I^{\PP},x_{12},x_{4(f_1+1)})=\big((x_1x_2^{b_1+k},\, x_2^{b_1+k}x_3, x_3^{e_1}x_4^{f_1})^{\PP}, x_{12},x_{4(f_1+1)}\big).
  \]
Due to \Cref{compntpolariz} and \Cref{proplin}, it follows that the ideal $(x_1x_2^{b_1+k},\, x_2^{b_1+k}x_3, x_3^{e_1}x_4^{f_1})$ has a linear resolution, which is a contradiction by \Cref{equivalent}. Therefore, $I$ has no linear resolution.
\end{proof}

\begin{lemma}\label{lem:whiskers}
Let $I=(x_1^{a_1}x_2^{b_1}, \ldots, x_1^{a_1-k}x_2^{b_1+k}, x_2^{c_1}x_3^{d_1},\ldots, x_2^{c_1-l}x_3^{d_1+l}, x_3^{e_1}x_4^{f_1},\ldots, x_3^{e_1-t}x_4^{f_1+t},\\ x_2^{g_1}x_4^{h_1},\ldots,x_2^{g_1-s}x_4^{h_1+s})$
be a support-two monomial ideal whose base graph is obtained from the complete graph $K_3$ by attaching a whisker at the vertex $x_2$. If $k\geq 1$, then $I$ has no linear resolution.
\end{lemma}
\begin{proof}
Let us assume $I$ has a linear resolution. Then $I_1=(u\in \G(I)\mid x_3\nmid u)$ and $I_2=(u\in \G(I)\mid x_4\nmid u)$ both have linear resolution by \Cref{proplin}. Note that $I_1$ and $I_2$ are support-two monomial ideals of the form given in \Cref{lem:p3}, and thus, we have
$$I=(x_1^{k+1}x_2^{b_1},\ldots,x_1x_2^{b_1+k}, x_2^{b_1+k}x_3, x_2^{b_1+k}x_4, x_3^{e_1}x_4^{f_1},\ldots, x_3^{e_1-t}x_4^{f_1+t}).$$
Now, we consider the ideal $I'=(u\in\G(I^{\PP})\mid x_{12}\nmid u\text{ and } x_{4(f_1+1)}\nmid u)$. Then we see that $I'=(x_1x_2^{b_1+k}, x_2^{b_1+k}x_3, x_2^{b_1+k}x_4,x_3^{e_1}x_4^{f_1})^{\PP}$. Since $I$ has linear resolution, by \Cref{compntpolariz} and \Cref{proplin} it follows that the minimal support-two monomial ideal 
$$(x_1x_2^{b_1+k}, x_2^{b_1+k}x_3, x_2^{b_1+k}x_4,x_3^{e_1}x_4^{f_1})$$ 
has a linear resolution. This fact contradicts \Cref{equivalent}. Hence, $I$ has no linear resolution.
\end{proof}  

\begin{proposition}\label{Lem:k3}
Let $I$ be a support-two monomial ideal generated in degree $3$ with the base graph $K_3$. Then the following hold:

\begin{enumerate}
   \item[\rm(1)] The ideal $I=(x_1^2x_2,x_1x_2^2,x_2^2x_3,x_1^2x_3)$
    has a linear resolution.
    
    \item[\rm(2)] The ideal $I=(x_1^2x_2,x_1x_2^2,x_2^2x_3,x_2x_3^2,x_1^2x_3,x_1x_3^2)$
    does not have a linear resolution.

    \item[\rm(3)] The ideal $I=(x_1^2x_2,x_1x_2^2,x_2^2x_3,x_1^2x_3,x_1x_3^2)$
    does not have a linear resolution.

    \item[\rm(4)] The ideal $I=(x_1^2x_2,x_1x_2^2,x_2x_3^2,x_1x_3^2)$
    does not have a linear resolution.

    \item[\rm(5)]
    The ideal $I=(x_1^2x_2,x_1x_2^2,x_2^2x_3,x_1x_3^2)$
    does not have a linear resolution.

\end{enumerate}
\end{proposition}
\begin{proof}
\underline{Proof of (1):} We can write $I=J + K$, where $J=(x_1^2x_2,x_1x_2^2)$ and $K=(x_1^2x_3,x_2^2x_3)$. Note that $J$ has a linear resolution by \Cref{lem:k2} and $K$ has a linear resolution by \Cref{equivalent}. Therefore, due to \cite[Corollary 2.4]{fhv09}, $I$ has a Betti splitting. Hence, we have $\reg(I)=\max\{\reg(J), \reg(K), \reg(J \cap K)-1\}$ by \cite[Corollary 2.2(1)]{fhv09}. Now, $J \cap K=x_1x_2x_3(x_1,x_2)$, which gives $\reg(J \cap K)=4$. Since $\reg(J)=\reg(K)=3$, we have $\reg(I)=3$. Thus, $I$ has a linear resolution.
\\
 \underline{Proof of (2):} Let us write $I=J + K$, where $J=(x_1^2x_2,x_1x_2^2,x_2^2x_3,x_1^2x_3)$ and $K=(x_1x_3^2,x_2x_3^2)$. Then by the Proof of (1), $J$ has a linear resolution, and by \Cref{equivalent}, $K$ has a linear resolution. Therefore, it follows from \cite[Corollary 2.4]{fhv09} that $I=J+K$ is a Betti splitting. Consequently, by \cite[Corollary 2.2(1)]{fhv09}, 
 $$\reg(I)=\max\{\reg(J), \reg(K), \reg(J \cap K)-1\}.$$
 Note that $J \cap K=(x_3^2(x_1^2,x_2^2))$, and thus, using \Cref{reg-sum} and \Cref{prop:reg ci}, we obtain $\reg(J \cap K)=5$. Since $\reg(J)=\reg(K)=3$, we have $\reg(I)=4$. Thus, $I$ has no linear resolution. \\
 One can similarly show that ideals of $(3), (4), (5)$ do not have linear resolution. Indeed, for the ideal $I$ mentioned in (3), one can take $J=(x_1^2x_2,x_1x_2^2,x_2^2x_3)$ and $K=(x_1^2x_3,x_1x_3^2)$; for the ideal $I$ mentioned in (4), one can take $J=(x_1^2x_2,x_1x_2^2)$ and $K=(x_2x_3^2,x_1x_3^2)$; for the ideal $I$ mentioned in (5), one can take $J=(x_1^2x_2,x_1x_2^2,x_2^2x_3)$ and $K=(x_1x_3^2)$.
\end{proof}

\begin{lemma}\label{prop:k3}
 Let $$I=(x_1^{a_1}x_2^{b_1},\ldots, x_1^{a_1-k}x_2^{b_1+k}, x_2^{c_1}x_3^{d_1}, \ldots,x_2^{c_1-l}x_3^{d_1+l},x_3^{e_1}x_1^{f_1},\ldots,x_3^{e_1-t}x_1^{f_1+t} )$$ be a support-two monomial ideal with the base graph $K_3$ and $k\geq 1$. Then $I$ has a linear resolution if and only if $l=t=0$, $d_1=e_1=1$, $b_1=a_1-k=1$ and $c_1=f_1=k+1$, i.e. $I=(x_1^{k+1}x_2,\ldots,x_1x_2^{k+1}, x_2^{k+1}x_3, x_1^{k+1}x_3)$.
\end{lemma}

\begin{proof}
Suppose that $I$ has a linear resolution. Then $I$ should be equigenerated in degree $d$ and $\reg(I)=d$. Thus, we have $a_1+b_1=c_1+d_1=e_1+f_1=d$. Since $k\geq 1$, we have $d\geq3$. For $d=3$, the desired conclusion follows from \Cref{Lem:k3} as the ideals listed in \Cref{Lem:k3} constitute all five possible (up to the permutation of variables) non-minimal tight support-two monomial ideals with the base graph $K_3$ and generated in degree $3$. Thus, we may assume $d \geq 4$.  Now, we proceed by mathematical induction on the number of minimal generators of $I$, that is, on $\mu(I)$. Also, we have $k \geq 1$. Then $x_1^{a_1-1}x_2^{b_1+1} \in \mathcal{G}(I)$. First, consider the following two cases.\\
\noindent\textbf{Case-I.} Let $c_1-l \leq b_1$. In this case, we take the monomial ideal $J$ such that $J^{\PP}=(u\in \G(I^{\PP})\mid x_{2(b_1+1)}\nmid u)$. Then one can verify that 
$$J=(x_1^{a_1}x_2^{b_1}, x_2^{c_1-l}x_3^{d_1+l}, x_2^{c_1-l+1}x_3^{d_1+l-1}, \ldots, x_2^{\min\{b_1,c_1\}}x_3^{\max\{a_1,d_1\}}, x_3^{e_1}x_1^{f_1},\ldots, x_3^{e_1-t}x_1^{f_1+t}).$$
Since $I$ has a linear resolution, we obtain by applying \Cref{compntpolariz} and \Cref{proplin} that $J$ also has a linear resolution. Now, $k\geq 1$ implies $\mu(J)<\mu(I)$. Since the base graph of $J$ is $K_3$, $J$ can not be a minimal support-two monomial ideal. Thus, by the induction hypothesis and the fact that $a_1>1$ (as $d\geq 4$ and $k\geq 1$), we have 
$$J=(x_1^{a_1}x_2^{b_1}, x_2^{c_1-l}x_3^{d_1+l},x_3^{e_1}x_1^{f_1}, \ldots, x_3^{e_!-t}x_1^{f_1+t}),$$
where $b_1=c_1-l=1$, $a_1=f_1+t$, $e_1=d_1+l$ and $t\geq 1$ (because of the fact that $J$ is not a minimal support-two monomial ideal).\\
\noindent \textbf{Case-II.} Let $c_1-l>b_1$. Again, we consider the monomial ideal $J$ as defined in Case-I. In this case, we see that
$$J=(x_1^{a_1}x_2^{b_1}, x_3^{e_1}x_1^{f_1}, \ldots, x_3^{e_1-t}x_1^{f_1+t}).$$
Note that $J$ has a linear resolution by \Cref{compntpolariz} and \Cref{proplin} as $I$ has so. Then, using \Cref{equivalent} (if $\mu(J)=\mu(\sqrt{J})$) and \Cref{lem:p3} (if $\mu(J)>\mu(\sqrt{J})$), we obtain that $b_1=1$, $a_1=f_1+t$, and $e_1-t=1$.\\
Combining both the above cases, we surely get $a_1=f_1+t$ and $b_1=1$. Similarly, considering the ideal $K$ such that $K^{\PP}=(u\in\G(I^{\PP})\mid x_{1(a_1-k+1)}\nmid u)$, one can obtain that $a_1-k=1$, $b_1+k=c_1$. Therefore, we have
$$I=(x_1^{k+1}x_2, \ldots,x_1x_2^{k+1}, x_2^{k+1}x_3,\ldots,x_2^{k+1-l}x_3^{l+1},x_3^{1+t}x_1^{k+1-t}, \ldots, x_3x_1^{k+1}).$$
Now, it remains to show that $t=l=0$. Note that $d\geq 4$ implies $k \geq 2$. Now, we consider the ideal $L$ such that $L^{\PP}=(u\in \G(I^{\PP})\mid x_{2(k+1)}\nmid u)$. Since $I$ has a linear resolution, $L$ also has a linear resolution by \Cref{compntpolariz} and \Cref{proplin}. Clearly, we have $\mu(L)<\mu(I)$. Since $k\geq 2$, by the induction hypothesis, the only choice of $L$ is $(x_1^{k+1}x_2, \ldots, x_1^2x_2^k,x_3x_1^{k+1}, \ldots,x_3^{1+t}x_1^{k+1-t})$, i.e. $l=0$. Now, $L$ is a non-minimal tight support-two monomial ideal with the base graph $P_3$. Thus, by \Cref{lem:p3}, we have $t=0$. 
%Then we have 
%%$L=(x_1^{k+1}x_2, \ldots, x_1^2x_2^k, x_2^kx_3^2,\ldots, x_2^{k+1-l}x_3^{1+l}, x_3^{1+t}x_1^{k+1-t}x_1^{k+1-t}, \ldots, x_3x_1^{k+1}).$
%Since $k \geq 2$, then by the induction hypothesis, we obtain that $l=0$ which gives $(I^{\PP},x_{2(k+1)})=(x_1^{k+1}x_2, \ldots, x_1^2x_2^k,x_3x_1^{k+1}, \ldots,x_3^{1+t}x_1{k+1-t})^{\PP}$ has linear resolution. Then by \Cref{lem:p3}, $t=0$.

Conversely, assume $l=0$, $t=0$, $d_1=e_1=1$, $b_1=a_1-k=1$ and $c_1=f_1=k+1$. Then $I=(x_1^{k+1}x_2,\ldots, x_1x_2^{k+1}, x_2^{k+1}x_3,x_3x_1^{k+1})$. 
We can write $I=J+K$, where $J=(x_1^{k+1}x_2,\ldots, x_1x_2^{k+1},x_2^{k+1}x_3)$ and $K=(x_3x_1^{k+1})$. By \Cref{lem:p3}, $J$ has a linear resolution, and $K$ has a linear resolution as it is generated by one monomial. Then by \cite[Corollary 2.4]{fhv09}, $I=J+K$ is a Betti splitting. Thanks to \cite[Corollary 2.2(1)]{fhv09}, we have $\reg(I)=\max\{\reg(J), \reg(K), \reg(J \cap K)-1\}$. Observe that $J \cap K=(x_1^{k+1}x_2x_3)$, which gives $\reg(J\cap K)=k+3$. Since $\reg(J)=\reg(K)=k+2$, we have $\reg(I)=k+2$. Hence, $I$ has a linear resolution.
\end{proof}

% \begin{lemma}
%  Let 
%  \begin{align*}
%  I=&(x_1^{a_1}x_2^{b_1},x_1^{a_1+1}x_2^{b_2-1},\ldots,x_1^{a_1+k}x_2^{b_2-k},x_2^{c_1}x_3^{d_1},x_2^{c_1+1}x_3^{d_1-1}, \ldots, x_2^{c_1+l}x_3^{d_1-l} ,x_3^{e_1}x_4^{f_1}, \\ &
%  x_3^{e_1+1}x_4^{f_1-1}, \ldots, x_3^{e_1+m}x_4^{f_1-m}, x_4^{g_1}x_1^{h_1},x_4^{g_1+1}x_1^{h_1-1},\ldots,x_4^{g_1+n}x_1^{h_1-n})
%  \end{align*}
%  be a support-two monomial ideal with the base graph $C_4$ and $k \geq 1$. If $I$ has linear resolution, then $l=0$, $c_1=b_1+k$, $d_1=1$ and $a_1-k=1$, i.e. $$I=(x_1^{k+1}x_2^{b_1},\ldots,x_1x_2^{b_1+k},x_2^{b_1+k}x_3,x_3x_4^{g_1+n},x_4^{g_1+n}x_1).$$
% \end{lemma}
% \begin{proof}
% Suppose $I$ has linear resolution. Then Note that 
% \end{proof}

\begin{figure}[ht]
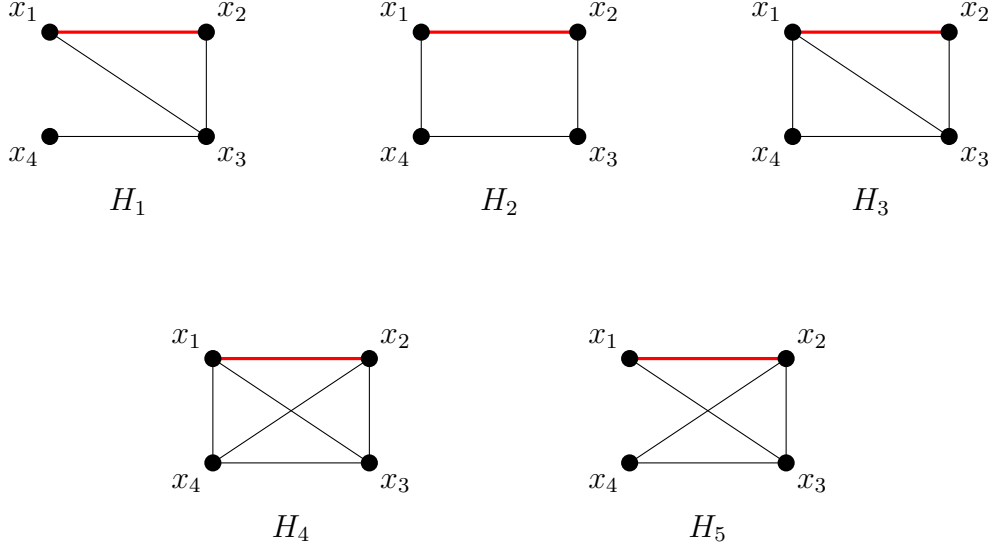

    \centering

    % First row
    \basegraph{1/2,2/3,1/3,3/4}{$H_1$}
    \hspace{1.2cm}
    \basegraph{1/2,2/3,3/4,4/1}{$H_2$}
    \hspace{1.2cm}
    \basegraph{1/2,2/3,3/4,4/1,1/3}{$H_3$}

    \vspace{1.2cm}

    % Second row
    \basegraph{1/2,2/3,3/4,4/1,1/3,2/4}{$H_4$}
    \hspace{1.8cm}
    \basegraph{1/2,1/3,2/3,2/4,3/4}{$H_5$}

    \caption{Forbidden base graphs on the vertex set
    $\{x_1,x_2,x_3,x_4\}$.}
    \label{fig-2}
\end{figure}
%------------------------------------------------------------
\begin{lemma}\label{lem:no-LR}
Let $I$ be a tight support-two monomial ideal whose base graph is one of the
forms {$H_1$, $H_2$, $H_3$, $H_4$, $H_5$} in \Cref{fig-2}. Suppose
\[
  \bigl|\{u\in\G(I)\mid \supp(u)=\{x_1,x_2\}\}\bigr|\geq 2 .
\]
Then $I$ has no linear resolution.
\end{lemma}

\begin{proof}
If possible, let $I$ have a linear resolution. For $1\le i\le 4$, we define the ideal
\[
  I_{x_i}:=\{u\in\G(I)\mid x_i\nmid u\}.
\]
By \Cref{proplin}, $I_{x_i}$ has a linear resolution for each $i$. Note that the ideal $I_{x_4}$ has the base graph either $K_3$ or $P_3$.
Then, applying \Cref{lem:p3} if the base graph of $I_{x_4}$ is $P_3$, or applying \Cref{prop:k3} if the base graph of $I_{x_4}$ is $K_3$, we get
\[
  A_{12}:=\{u\in\G(I)\mid \supp(u)=\{x_1,x_2\}\}
  =\{x_1^{k+1}x_2,\ \dots,\ x_1x_2^{k+1}\} \text{ for some } k\ge 1, \text{ and}
\]
\[A_{23}:=\{u\in\G(I)\mid \supp(u)=\{x_2,x_3\}\}=\{x_3x_2^{k+1}\}.\] 
Now, we consider two cases. 

\smallskip
\noindent\textbf{Case I:} Suppose that the base graph of $I$ is one of the forms
{$H_1, H_2, H_3$}.
Then the ideal $I_{x_1}$ has base graph $P_3$. Since $I_{x_1}$ has a linear
resolution and $A_{23}=\{x_3x_2^{k+1}\}$ with $k\ge1$, \Cref{equivalent} and \Cref{lem:p3} give a contradiction.\\
\noindent\textbf{Case II:} Suppose that the base graph of $I$ is one of the forms {$H_4$, $H_5$}. Then the ideal $I_{x_1}$ has base graph $K_3$. Since $I_{x_1}$ has a linear resolution and $A_{23}=\{x_3x_2^{k+1}\}$, by \Cref{equivalent} and \Cref{lem:p3} the only
possibility is
\[
  I_{x_1}=\{\,x_3x_2^{k+1},\ x_3x_4^{k+1},\ x_2^{k+1}x_4,\ \dots,\ x_2x_4^{k+1}\,\}.
\]
% NOTE: for the base graph (4) = K_4 the ideal I_{x_3} below would also contain the
% generators supported on {x_1,x_4}; the handwritten notes omit them.
Therefore, $I_{x_3}=\{\,x_1^{k+1}x_2,\ \dots,\ x_1x_2^{k+1},\ x_2^{k+1}x_4,\ \dots,\ x_2x_4^{k+1}\,\}.$ Since $I_{x_3}$ has a linear resolution with base graph $P_3$ and $k\ge1$, \Cref{lem:p3} gives a contradiction.\par
Hence, $I$ does not have any linear resolution.
\end{proof}

\begin{theorem}\label{thm:linearresolution}
Let $I$ be a non-minimal support-two monomial ideal, i.e. $\mu(I)>\mu(\sqrt{I})$. Then $I$ has a linear resolution if and only if $I$ is one of the following ideals (up to the permutation of variables)
\[
\begin{aligned}
\textup{(1)}\quad
I &=
(x_1^{a_1}x_2^{b_1}, \ldots, x_1^{a_1-k}x_2^{b_1+k});\\[1ex]
\textup{(2)}\quad
I &=
(x_1^{a_1}x_2^{b_1}, \ldots, x_1x_2^{b_1+k},
x_2^{b_1+k}x_3, \ldots, x_2^{b_1+k}x_n);\\[1ex]
\textup{(3)}\quad
I &=
(x_1^{k}x_2,\ldots,x_1x_2^{k},
x_3x_1^{k},\ldots,x_sx_1^{k},
x_{s+1}x_2^{k},\ldots,x_nx_2^{k});\\[1ex]
\textup{(4)}\quad
I &=
(x_1^{k}x_2,\ldots,x_1x_2^{k},
x_3x_1^{k},\ldots,x_nx_1^{k},
x_{3}x_2^{k},\ldots,x_nx_2^{k});\\[1ex]
\textup{(5)}\quad
I &=
(x_1^{k}x_2,\ldots,x_1x_2^{k},
x_3x_1^{k},\ldots,x_px_1^{k},
x_{p+1}x_1^{k},\ldots,x_nx_1^{k},\\
&\qquad\qquad\qquad\qquad\qquad\qquad\qquad\qquad\quad x_{p+1}x_2^{k},\ldots,x_nx_2^{k});\\[1ex]
\textup{(6)}\quad
I &=
(x_1^{k}x_2,\ldots,x_1x_2^{k},
x_3x_1^{k},\ldots,x_px_1^{k},
x_{p+1}x_1^{k},\ldots,x_sx_1^{k},\\
&\qquad\qquad\qquad\qquad\qquad
x_{p+1}x_2^{k},\ldots,x_sx_2^{k},
x_{s+1}x_2^{k},\ldots,x_nx_2^{k}).
\end{aligned}
\]
\end{theorem}
\begin{proof}
Suppose $I$ has a linear resolution. Then, by \Cref{proplin}(1), $I_{[H]}$ has a linear resolution, where $H$ is an induced subgraph of $G_I$ and $I_{[H]}=\left(f\in\G(I)\mid \supp(f)\in E(H)\right)$. Since $I$ is equigenerated and $\alpha(I)>2$, we have $\alpha(I_{[H]})>2$. Given that $I$ is not a minimal support-two monomial ideal, i.e. $\mu(I)>\mu(\sqrt{I})$. Then without loss of generality, we can assume $|\{u\in \G(I)\mid \supp(u)=\{x_1,x_2\}\}|\geq 2$. Also, due to \Cref{proplin}(1) and \Cref{lem:k2}, the ideal $I$ should be a tight support-two monomial ideal. Hence, using \Cref{lem:k2}, \ref{lem:p3}, \ref{lem:p4}, \ref{lem:whiskers}, \ref{prop:k3}, \ref{lem:no-LR}, we can conclude that $I$ is one of the forms (up to the permutation of the variables) (1), (2), (3), (4), (5), (6); see \Cref{fig-3} for the corresponding base graphs.\par 

Conversely, let $I$ be the ideal of the form (1) to (6). Now, observe that the ideals of the form (1) to (5) can be obtained from the ideal of the form (6) by removing all minimal generators divisible by some variables. Thus, due to \Cref{proplin}(1), it is enough to show that the ideal $I$ of the form given in (6) has a linear resolution. First, we split the $I$ as $I=J+K$, where
\begin{align*}
    J=\{x_1^{k}x_2, \dots, x_1x_2^{k},\ x_3x_1^{k}, \dots, x_sx_1^{k}\}, \text{ and }
    K=\{x_{p+1}x_2^{k}, \dots, x_nx_2^{k}\}.
\end{align*}
Now, we have
\[
  J:x_1=\{x_1^{k-1}x_2, \dots, x_1x_2^{k-1}, x_2^{k}, x_3x_1^{k-1}, \dots, x_sx_1^{k-1}\},
\]
and
\begin{align*}
  (J:x_1):x_2&=\{\,x_1^{k-1},\ x_1^{k-2}x_2,\ \dots,\ x_1x_2^{k-2},\ x_2^{k-1}\,\}\\
  &= (x_1,x_2)^{k-1}.
\end{align*}
Note that $k\geq 2$ as $\alpha(I)>2$. Therefore, by \Cref{prop:reg ci}, $\reg((J:x_1):x_2)=k-1$. Again, note that $((J:x_1),x_2)=(x_3x_1^{k-1},\ldots,x_sx_1^{k-1},x_2)$, which gives $\reg((J:x_1),x_2)=k$ by \Cref{reg-sum} and \Cref{equivalent}. Hence, considering the exact sequence
\[0 \rightarrow R/((J:x_1):x_2)(-1) \rightarrow R/(J:x_1) \rightarrow R/((J:x_1),x_2) \rightarrow 0.\]
and using \Cref{lem:regularity-lemma}, we conclude that $\reg(J:x_1)=k$. Since $(J,x_1)=(x_1)$, we have $\reg(J,x_1)=1$. Thus, applying \Cref{lem:regularity-lemma} to the following exact sequence
\[0 \rightarrow R/(J:x_1)(-1) \rightarrow R/J \rightarrow R/(J,x_1), \rightarrow 0,\]
we obtain that $\reg(J)=k+1$. Consequently, $J$ has a linear resolution. Also, it is clear from \Cref{equivalent} that the ideal $K$ has a linear resolution. Therefore, $I=J+K$ is a Betti splitting by \cite[Corollary 2.4]{fhv09}. Now, one can easily compute that
\[ 
J\cap K=x_1(x_{p+1}x_2^{k},\ \dots,\ x_nx_2^{k}).
\]
Thus, by \cref{reg-sum} and \Cref{equivalent}, we have $\reg(J\cap K)=k+2$.
Consequently, by \cite[Corollary 2.2(1)]{fhv09}, we have
\begin{align*}
  \reg(I)&=\max\{\reg(J), \reg(K), \reg(J \cap K)-1\}\\
  &=\{k+1,k+1,(k+2)-1\}\\
  &=k+1.
\end{align*}
Hence, $I$ has a linear resolution as $I$ is equigenerated in degree $k+1$.
\end{proof}

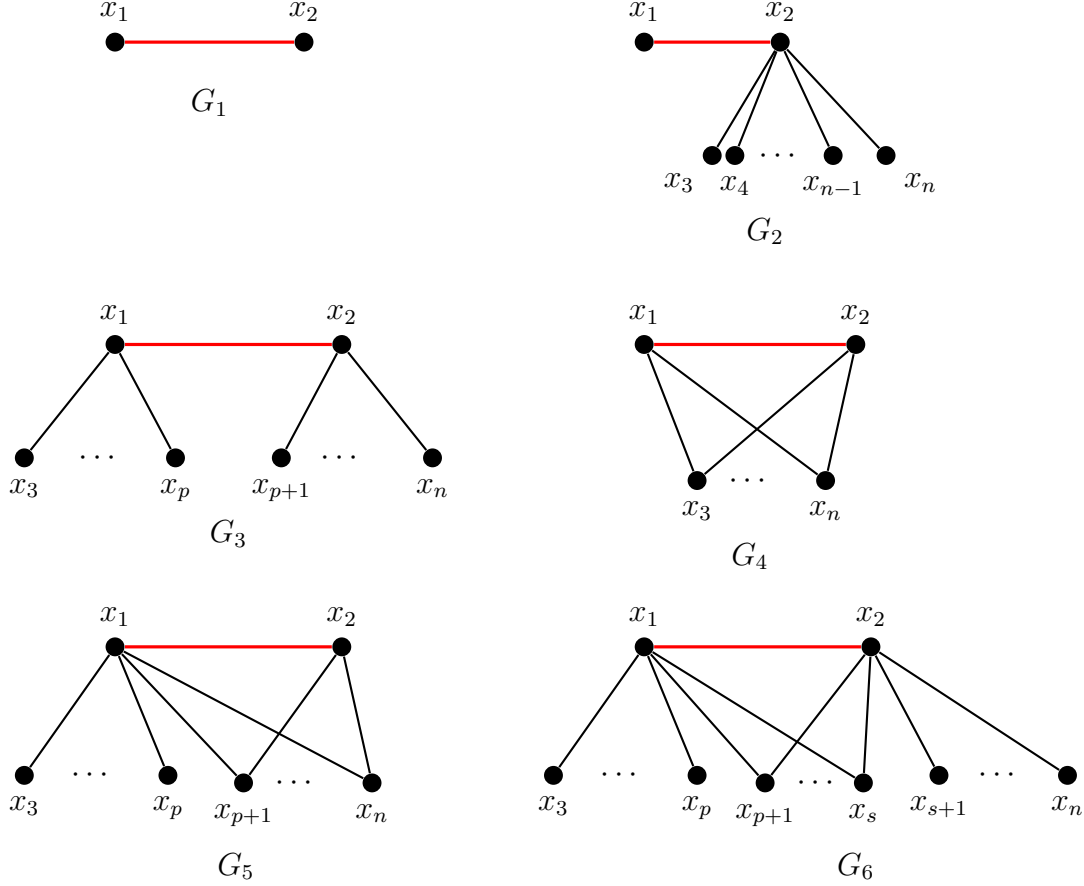
\begin{figure}[htbp]
\centering
\begin{tikzpicture}[
    scale=1,
    transform shape,
    vtx/.style={
        circle,
        fill=black,
        inner sep=0pt,
        minimum size=2.5mm
    },
    ed/.style={thick},
    dots/.style={
        inner sep=0pt,
        font=\normalsize
    },
     lbl/.style={font=\small},
    red edge/.style={red, very thick}
]

% =========================================================
% G_1
% =========================================================
\begin{scope}[xshift=-5.2cm,yshift=4.0cm]
\node[vtx,label=above:$x_1$] (x1) at (0,0) {};
\node[vtx,label=above:$x_2$] (x2) at (25mm,0) {};
\draw[red edge] (x1)--(x2);
\node at (12.5mm,-8mm) {$G_1$};
\end{scope}

% =========================================================
% G_2
% =========================================================
\begin{scope}[xshift=1.8cm,yshift=4.0cm]
\node[vtx,label=above:$x_1$] (x1) at (0,0) {};
\node[vtx,label=above:$x_2$] (x2) at (18mm,0) {};
\draw[red edge] (x1)--(x2);

\node[vtx,label=below left:$x_3$]  (x3)  at (9mm,-15mm) {};
\node[vtx,label=below:$x_4$]       (x4)  at (12mm,-15mm) {};
\node[dots] at (18mm,-15mm) {$\cdots$};
\node[vtx,label=below:$x_{n-1}$] (xn1) at (25mm,-15mm) {};
\node[vtx,label=below right:$x_n$] (xn) at (32mm,-15mm) {};

\foreach \v in {x3,x4,xn1,xn}
    \draw[ed] (x2)--(\v);

\node at (16mm,-25mm) {$G_2$};
\end{scope}

% =========================================================
% G_3
% =========================================================
\begin{scope}[xshift=-5.2cm,yshift=0cm]
\node[vtx,label=above:$x_1$] (x1) at (0,0) {};
\node[vtx,label=above:$x_2$] (x2) at (30mm,0) {};
\draw[red edge] (x1)--(x2);

\node[vtx,label=below:$x_3$] (x3) at (-12mm,-15mm) {};
\node[dots] at (-2mm,-15mm) {$\cdots$};
\node[vtx,label=below:$x_p$] (xp) at (8mm,-15mm) {};

\foreach \v in {x3,xp}
    \draw[ed] (x1)--(\v);

\node[vtx,label=below:$x_{p+1}$] (xp1) at (22mm,-15mm) {};
\node[dots] at (30mm,-15mm) {$\cdots$};
\node[vtx,label=below:$x_n$] (xn) at (42mm,-15mm) {};

\foreach \v in {xp1,xn}
    \draw[ed] (x2)--(\v);

\node at (15mm,-25mm) {$G_3$};
\end{scope}

% =========================================================
% G_4
% =========================================================
\begin{scope}[xshift=1.8cm,yshift=0cm]
\node[vtx,label=above:$x_1$] (x1) at (0,0) {};
\node[vtx,label=above:$x_2$] (x2) at (28mm,0) {};
\draw[red edge] (x1)--(x2);

\node[vtx,label=below:$x_3$] (x3) at (7mm,-18mm) {};
\node[dots] at (14mm,-18mm) {$\cdots$};
\node[vtx,label=below:$x_n$] (xn) at (24mm,-18mm) {};

\foreach \v in {x3,xn}{
    \draw[ed] (x1)--(\v);
    \draw[ed] (x2)--(\v);
}

\node at (14mm,-28mm) {$G_4$};
\end{scope}

% =========================================================
% G_5
% =========================================================
\begin{scope}[xshift=-5.2cm,yshift=-4.0cm]
\node[vtx,label=above:$x_1$] (x1) at (0,0) {};
\node[vtx,label=above:$x_2$] (x2) at (30mm,0) {};
\draw[red edge] (x1)--(x2);

\node[vtx,label=below:$x_3$] (x3) at (-12mm,-17mm) {};
\node[dots] at (-3mm,-17mm) {$\cdots$};
\node[vtx,label=below:$x_p$] (xp) at (7mm,-17mm) {};

\foreach \v in {x3,xp}
    \draw[ed] (x1)--(\v);

\node[vtx,label=below:$x_{p+1}$] (xp1) at (17mm,-18mm) {};
\node[dots] at (24mm,-18mm) {$\cdots$};
\node[vtx,label=below:$x_n$] (xn) at (34mm,-18mm) {};

\foreach \v in {xp1,xn}{
    \draw[ed] (x1)--(\v);
    \draw[ed] (x2)--(\v);
}

\node at (16mm,-29mm) {$G_5$};
\end{scope}

% =========================================================
% G_6
% =========================================================
\begin{scope}[xshift=1.8cm,yshift=-4.0cm]
\node[vtx,label=above:$x_1$] (x1) at (0,0) {};
\node[vtx,label=above:$x_2$] (x2) at (30mm,0) {};
\draw[red edge] (x1)--(x2);

\node[vtx,label=below:$x_3$] (x3) at (-12mm,-17mm) {};
\node[dots] at (-3mm,-17mm) {$\cdots$};
\node[vtx,label=below:$x_p$] (xp) at (7mm,-17mm) {};

\foreach \v in {x3,xp}
    \draw[ed] (x1)--(\v);

\node[vtx,label=below:$x_{p+1}$] (xp1) at (16mm,-18mm) {};
\node[dots] at (23mm,-18mm) {$\cdots$};
\node[vtx,label=below:$x_s$] (xs) at (29mm,-18mm) {};

\foreach \v in {xp1,xs}{
    \draw[ed] (x1)--(\v);
    \draw[ed] (x2)--(\v);
}

\node[vtx,label=below:$x_{s+1}$] (xs1) at (39mm,-17mm) {};
\node[dots] at (47mm,-17mm) {$\cdots$};
\node[vtx,label=below:$x_n$] (xn) at (56mm,-17mm) {};

\foreach \v in {xs1,xn}
    \draw[ed] (x2)--(\v);

\node at (28mm,-29mm) {$G_6$};
\end{scope}

\end{tikzpicture}
\caption{Base graphs of linear non-minimal support-two monomial ideals}
\label{fig-3}
\end{figure}

\subsection*{Acknowledgements} The second author is partially supported by the Anusandhan National Research Foundation (ANRF), Government of India, under the ARG-MATRICS Grant No. ANRF/ARGM/2025/002203/MTR.

\bibliographystyle{plain}

\bibliography{reference}
\end{document}